\documentclass[11pt]{amsart}

\usepackage{amsmath,amsthm,amsfonts, bbm, amssymb, hyperref,graphicx,color}

\newcommand{\ii}{\mathbbm{i}}
\newcommand{\C}{\mathbb{C}}
\newcommand{\CP}{\mathbb{CP}}
\newcommand{\Hyp}{\mathbb{H}}
\newcommand{\Heis}{\mathbf{H}}
\newcommand{\Sieg}{\mathfrak S}
\newcommand{\D}{\operatorname D}
\newcommand{\R}{\mathbb{R}}
\newcommand{\Z}{\mathbb{Z}}
\newcommand{\N}{\mathbb{N}}
\newcommand{\Q}{\mathbb{Q}}
\newcommand{\BA}{\mathbf{BA}}
\newcommand{\BD}{\mathbf{BD}}
\newcommand{\Carnot}{\mathfrak C}

\newcommand{\norm}[1]{\left\vert #1 \right \vert}	
\newcommand{\Norm}[1]{\left\Vert #1 \right \Vert}

\renewcommand{\Re}{\text{Re}}
\renewcommand{\Im}{\text{Im}}

\newcommand{\height}{\operatorname{ht}}

\def\[#1\]{\begin{align}#1\end{align}}
\def\(#1\){\begin{align*}#1\end{align*}}

\newtheorem{thm}{Theorem}

\newtheorem{claim}[thm]{Claim}
\newtheorem{prop}[thm]{Proposition}
\newtheorem{lemma}[thm]{Lemma}
\newtheorem{cor}[thm]{Corollary}
\newtheorem{question}[thm]{Question}

\theoremstyle{definition}
\newtheorem{defi}[thm]{Definition}

\theoremstyle{definition}
\newtheorem{example}[thm]{Example}
\newtheorem{remark}[thm]{Remark}

\numberwithin{equation}{section}
\numberwithin{thm}{section}

\newcommand{\st}{\;:\;}

\title[Diophantine: Carnot and Siegel]{Intrinsic Diophantine approximation in Carnot groups and in the Siegel model of the Heisenberg group}
\author[A. Lukyanenko]{Anton Lukyanenko}
\address{
Department of Mathematics\\
University of Michigan\\
530 Church Street\\
Ann Arbor, MI 48109}
\email{anton@lukyanenko.net}
\author[J. Vandehey]{Joseph Vandehey}
\address{
Department of Mathematics\\
University of Georgia at Athens\\
Athens, GA 30602
}
\email{vandehey@uga.edu}

\subjclass[2010]{Primary  11J83, Secondary 22E25, 11J70, 53C17}

\keywords{Badly approximable, Carnot group, Continued fractions, Diophantine approximation, Heisenberg group, Schmidt games}

\begin{document}

\begin{abstract}
We initiate the study of an intrinsic notion of Diophantine approximation on a rational Carnot group $G$. If $G$ has Hausdorff dimension $Q$, we show that its  Diophantine exponent is equal to $(Q+1)/Q$, generalizing the case $G=\R^n$. We furthermore obtain a precise asymptotic on the count of rational approximations.

We then focus on the case of the Heisenberg group $\Heis^n$, distinguishing between two notions of Diophantine approximation by rational points in $\Heis^n$: Carnot Diophantine approximation and Siegel Diophantine approximation.

After computing the Siegel Diophantine exponent (surprisingly, equal to 1 for all $\Heis^n$), we consider Siegel-badly approximable points to show that Siegel approximation is linked to both Heisenberg continued fractions and to geodesics in complex hyperbolic space.

We conclude by showing that Carnot and Siegel approximation are qualitatively different: Siegel-badly approximable points are Schmidt winning in any complete Ahlfors regular subset of $\Heis^n$, while the set of Carnot-badly approximable points does not have this property.
\end{abstract}

\date{\today}

\maketitle

\section{Introduction}

\subsection{Carnot groups and Diophantine approximation} 
Let $G$ be a rational Carnot group homeomorphic to $\R^n$ and $G(\Z)\subset G$ the integer lattice (see \S \ref{sec:Carnot} for definitions). Diophantine-style questions concerning $G$ arise from connections to complex hyperbolic geometry \cite{MR1844559, MR1863852}, and from dynamics on the nilmanifold $\Lambda \backslash G$ \cite{MR3357181, MR2586348,  MR2877065}, see \S \ref{sec:otherapproaches}.

In the spirit of Gromov's ``Carnot-Carath\'eodory spaces seen from within'' \cite{MR1421823}, we initiate a study of Diophantine approximation that is based directly on the structure of $G$: the integer lattice $G(\Z)$, the dilations $\delta_q$, rational points $G(\Q)=\{\delta_q^{-1} p: p\in G(\Z), q\in \N\}$, and Carnot metric $d$. Generalizing the case $G=\R^n$, we first show (Theorems \ref{thm:GeometricUpperBound}  and \ref{thm:GeometricLowerBound}):
\begin{thm}[Carnot Diophantine exponent]
\label{thm:CarnotDiophantine}
Suppose $G$ has Hausdorff dimension $Q$. Then, for almost every $g\in G$ the Carnot Diophantine equation
\begin{equation}
\label{CarnotDiophantine}
d(g, \delta_q^{-1} p) < q^{-\alpha}, \text{ for }q\in \N, p\in G(\Z)
\end{equation}
has infinitely many solutions for $\alpha \leq (Q+1)/Q =:\alpha(G)$. On the other hand, if $\alpha > \alpha(G)$ then the set of points admitting infinitely many solutions to \eqref{CarnotDiophantine} has measure zero.
\end{thm}

In the particular case $G=\R^n$, it follows from the Minkowski Linear Form Theorem that the Diophantine exponent $\alpha(\R^n) = (n+1)/n$ can be realized for \emph{all} points: if $\epsilon>0$ then \emph{every} $\vec x\in \R^n$ admits infinitely many rational approximations $q^{-1}\vec p$ with $\norm{\vec x - q^{-1} \vec p}< q^{-(\alpha(\R^n)-\epsilon)}$.  We show that the result cannot be generalized to all rational Carnot groups: if $G\neq \R^n$ then it contains a linear subset with a strictly higher Diophantine exponent, as well as one with a strictly lower Diophantine exponent, see Theorem \ref{thm:CarnotWeirdness}.

Strengthening Theorem \ref{thm:CarnotDiophantine}, we are able to count solutions to the Diophantine problem in Carnot groups (we obtain a more precise asymptotic for supremal metrics on $G$, see Theorem \ref{thm:GeometricLowerBound}):

\begin{thm}[Count of rational approximates]
\label{thm:IntroApproximates}
Let $G$ be a rational Carnot group with Diophantine exponent $\alpha(G)$. Then there exists $C>0$ such that for almost every point $g\in G$, the number of solutions to 
$$
d(g, \delta_q^{-1} p) < q^{-\alpha(G)}, \text{ for }q\in \N, p\in G(\Z)
$$
satisfying $q<q_0$ is comparable to $\log q_0$. 
\end{thm}

\subsection{Definitions: Heisenberg group and Diophantine approximation}
\label{sec:CarnotSiegelDiophantus}
We next focus on the particular case of the simplest non-commutative Carnot group: the Heisenberg group $\Heis^n$. The group $\Heis^n$ is homeomorphic to $\R^{2n+1}$, is a two-step nilpotent group, and (with the Carnot metric) has Hausdorff dimension $Q=2n+2$. 

We will study two different notions of Diophantine approximations for $\Heis^n$, arising from two standard models of $\Heis^n$.

\begin{defi}[Carnot model]
\label{defi:CarnotHeisenberg}
In the \emph{Carnot model} (a.k.a.\ exponential or geometric model) $\Carnot^n$, one gives a point $h\in \Heis^n$ coordinates in $(\vec x, \vec y, t)\in \R^n\times \R^n \times \R$, and writes the group law as
$$(\vec x, \vec y, t)*(\vec x', \vec y',t') = (\vec x + \vec x', \vec y+\vec y', t+t'+2(\vec x \cdot \vec y' - \vec x' \cdot \vec y)).$$
While one can give the Heisenberg group a Riemannian metric, it is more natural to work with the \emph{gauge (or Cygan or Koranyi) metric}  defined by
$$d(h,h') = \Norm{h^{-1}*h} \quad \text{ and } \quad  \Norm{(x,y,t)}^4 = (x^2+y^2)^2+t^2.$$
The isomorphism $\delta_r(\vec x, \vec y, t) := (r \vec x, r \vec y, r^2t)$ dilates distances by factor $r$. We will also consider the \emph{infinity metric} $d_\infty$ defined by the gauge $\Norm{(x,y,t)}_\infty=\max\{\norm{\vec x}_\infty, \norm{\vec y}_\infty,  \frac{1}{2n} \norm{t}^{1/2}\}$. Both $d$ and $d_\infty$ are bi-Lipschitz equivalent to the Carnot (sub-Riemannian) metric, see \S \ref{sec:Carnot}.

Let $\Carnot^n(\Z)$ be the \emph{integer subgroup} of points in $\Carnot^n$ with integer coordinates, and likewise $\Carnot^n(\Q)$ the \emph{rational subgroup}. It is easy to see that every point of $\Carnot^n(\Q)$ is of the form $\delta_q^{-1}p$ for $q\in\N$ and $p\in \Carnot^n(\Z)$. If $q$ is minimal, then we set $\height_{\Carnot^n}(\delta_q^{-1}p):=q$. This is the \emph{Carnot height} of $\delta_q^{-1}p$.  

For each $\alpha>0$, the \emph{$\alpha$-Diophantine problem in $\Carnot^n$} asks to approximate a point $h\in \Carnot^n$ by a rational $\delta_q^{-1}p$ such that
$$d(h, \delta_q^{-1}p) \leq q^{-\alpha}.$$
\end{defi}

\begin{defi}[Siegel model]
The \emph{Siegel model} $\Sieg^n$ of $\Heis^n$ is the set of points $(\vec u, v)\in \C^n \times \C$ such that $2\Re(v)=\norm{\vec u}^2$. The group law in $\Sieg^n$ is 
$$(\vec u, v)*(\vec u', v') = (\vec u + \vec u', v+v'+\overline{\vec u}+u').$$
The gauge metric in $\Sieg^n$ is defined by
$$d(h,h') = \Norm{h^{-1}*h} \quad \text{ and } \quad  \Norm{(\vec u,v)} = \norm{v}^{1/2}.$$

We take $\Sieg^n(\Z):= \Sieg^n \cap \Z[\ii]^{n+1}$ and $\Sieg^n(\Q):= \Sieg^n \cap \Q[\ii]^{n+1}$.  The elements of $\Sieg^n(\Q)$ can be written as $q^{-1} \vec p$ for $q\in \Z[\ii]$ and $p\in \Z[\ii]^{n+1}$ (note that we do not require $\vec p$ to be in $\Sieg^n(\Z)$). 
If $\norm{q}$ is minimal, we set the \emph{Siegel height} to $\height_{\Sieg^n}(q^{-1}\vec p):=\norm{q}$.

For each $\alpha>0$, the \emph{$\alpha$-Diophantine problem in $\Sieg^n$} asks to approximate a point $h\in \Sieg^n$ by a rational $q^{-1}\vec p\in \Sieg^n$ such that
$$d(h, q^{-1}\vec p) \leq \norm{q}^{-\alpha}.$$
\end{defi}

The two models $\Carnot^n$ and $\Sieg^n$ are related by the isometry
\begin{align*}
f: &\Carnot^n \rightarrow \Sieg^n\\
&(\vec x, \vec y, t) \mapsto ((\vec x+\ii y)(1+\ii), t+\norm{\vec x}^2+\norm{\vec y}^2).
\end{align*}
One has $f(\Carnot^n(\Q))=\Sieg^n(\Q)$, so we can speak more abstractly of rational points $\Heis^n(\Q)$ in the Heisenberg group and of the two heights 
$\height_{\Carnot^n}$ and $\height_{\Sieg^n}$ of a rational point. No clear relationship is known between the two heights.

\begin{remark}The integer correspondence $f(\Carnot^1(\Z))=\Sieg^1(\Z)$ holds for in the first Heisenberg group, so we may speak only of $\Heis^1(\Z)$. However, we have to specify the model to speak of integer points in higher dimensions.
\end{remark}

More generally, we phrase the Diophantine approximation question as:
\begin{question}[General Diophantine approximation]
Consider a metric space $X$ with distance $d$, measure $\mu$, and some notion of rational points $X(\Q)$ where for each $x\in X(\Q)$ there is also a notion of a height $\height(x)$ for the point. The classical Diophantine approximation question asks, for a given $x\in X\setminus X(\Q)$ and positive, real-valued function $f$, how many rational points $y\in X(\Q)$ satisfy
\[
\label{generalDiophantine}
d(x,y) \le f(\height(y)).
\]
\end{question}

Suppose that \eqref{generalDiophantine} has infinitely many solutions for $f(z)=z^{-\alpha_0}$ for $\mu$-almost every point of $X$, but for any $\epsilon>0$, \eqref{generalDiophantine} has infinitely many solutions for  $f(z)=z^{-(\alpha_0-\epsilon)}$ for $\mu$-almost no points. We then write $\alpha(X,\height)=\alpha_0$ and refer to $\alpha_0$ as the \emph{Diophantine exponent of $X$}. 

If the height is clear from context, we simply write $\alpha(X)$.
If such an $\alpha(X)$ is known, we say that a point $x\in X$ is \emph{badly approximable} (is in the set $\BA_X$) if there exists $C>0$ such that for all $y\in G(\Q)$ one has $d(x,y) > C\height(y)^{\alpha(X)}$. 

We will refer to a constant $\alpha$ as the irrationality exponent of an individual point $x\in X$ if \eqref{generalDiophantine} has infinitely many solutions for $f(z)=z^{-\alpha}$ but only finitely many solutions for $f(z)=z^{-\alpha-\epsilon}$ for any $\epsilon>0$. The irrationality exponent of a point $x\in X$ can be larger or smaller than the Diophantine  exponent for $X$. Additionally, a subset $Y\subset X$ with measure $\mu'$ has Diophantine exponent $\alpha(Y,X)$ if $\mu'$-almost every point of $Y$ has irrationality exponent equal to $\alpha(Y,X)$.

\subsection{Siegel model and Diophantine approximation}

Theorem \ref{thm:CarnotDiophantine} provided the Diophantine exponent for all Carnot groups, and in particular for the Heisenberg group $\Heis^n$ with respect to the Carnot height $\height_{\Carnot^n}$. Adjusting the methods to count rational points in the space $\Sieg^n \subset \C^{n+1}$, we obtain results about rational approximation of points in $\Heis^n$ with respect to the Siegel height $\height_{\Sieg^n}$:

\begin{thm}[Siegel Diophantine Exponent, cf.\ \cite{MR1919402} and \S \ref{sec:otherapproaches}]
\label{thm:SiegelIntro}
For $\alpha=1$, and almost every $h\in \Heis^n$, the Siegel Diophantine equation
\begin{equation}
\label{eq:SiegelDiophantine}
d(h, (\vec p/q, r/q) ) < \norm{q}^{-\alpha}
\end{equation}
has infinitely many solutions $(\vec p/q, r/q)\in \Sieg^n(\Q)$. On the other hand, if $\alpha > 1$ then the set of points $h$ admitting infinitely many solutions to \eqref{eq:SiegelDiophantine} has measure zero.
\end{thm}

We next consider three sets $\BA_{\Sieg^1}, \BD, \D(\Gamma)\subset \Heis^1$ in the first Heisenberg group. The set $\BA_{\Sieg^1}$ is the set of badly approximable points with respect to the Siegel height (see \S \ref{sec:CarnotSiegelDiophantus}). The set $\BD$ is the set of points in $\Heis^1$ whose continued fraction expansion (defined by the authors in \cite{LV}, see \S \ref{sec:CFalpha}) is given by digits of bounded norm. The set $\D(\Gamma)$ for $\Gamma=U(2,1;\Z[\ii])$ is the set of endpoints of geodesics in complex hyperbolic space $\Hyp^2_\C$ whose image in the modular space $\Gamma \backslash \Hyp^2_\C$ is contained in some compact space. It is a classical fact that the analogous subsets of $\R$ coincide. We show:

\begin{thm}[Trichotomy for $\Heis^1$]
\label{thm:trichotomy}
Let $h\in \Heis^1$. Then the following are equivalent:
\begin{enumerate}
\item $h\in \BA_{\Sieg^1}$: it is badly approximable with respect to the Siegel height $\height_{\Sieg^1}$,
\item $h\in \BD$: its continued fraction digits are bounded in norm,
\item $h \in \D(\Gamma)$: it is the endpoint of a complex hyperbolic geodesic that is contained in some compact subset of $\Gamma \backslash \Hyp^2_\C$.
\end{enumerate}
\end{thm}

\subsection{Carnot vs Siegel}
We have so far provided Diophantine exponents $\alpha$ for approximation in the Heisenberg group $\Heis^n$ with respect to each the Carnot height $\height_{\Carnot^n}$ and the Siegel height $\height_{\Sieg^n}$. We have $\alpha(\Heis^n, \height_{\Carnot^n})=(2n+3)/(2n+2)$, while $\alpha(\Heis^n, \height_{\Sieg^n})=1$. In this sense, it is easier to solve the Carnot Diophantine problem. On the other hand, for $\Heis^1$, at least, we are able to find explicit solutions to the Siegel Diophantine problem using Heisenberg continued fractions.

We conclude by highlighting a qualitiative difference between Carnot Diophantine approximation and Siegel Diophantine approximation, which arises when one compares the badly approximable sets $\BA_{\Carnot^n}$ and $\BA_{\Sieg^n}$. 

\begin{thm}[Geometry of badly approximable points]
\label{thm:GeometryIntro}
The sets $\BA_{\Carnot^n}$ and $\BA_{\Sieg^n}$ are both non-empty and dense in $\Heis^n$. Furthermore, given any closed Ahlfors $\delta$-regular set $A\subset \Heis^n$, the set $\BA_{\Sieg^n}\cap A$ is Schmidt-winning in $A$, and has full Hausdorff dimension $\delta$ (which is also the Hausdorff dimension of $A$). On the other hand, $\BA_{\Carnot^n}$ avoids infinitely many lines through any rational point. In particular, $\BA_{\Carnot^n} \neq \BA_{\Sieg^n}$.
\end{thm}

\begin{remark}
In the case $n=1$ and $A=\Heis^1$, the fact that $\BA_{\Sieg^1}$ is winning in $\Heis^1$ follows from McMullen \cite{MR2720230} via Theorem \ref{thm:trichotomy}.
\end{remark}

\subsection{Other approaches}
\label{sec:otherapproaches}
In this paper, we approach Diophantine approximation on Carnot groups and in particular $\Heis^n$ from an intrinsic perspective, and employ direct number-theoretic arguments to establish our results. 

There is a large existing literature on Diophantine approximation on Riemannian manifolds (e.g.\ \cite{MR3357181, MR1844559, MR2586348, MR1863852,MR1919402, MR1652916, MR1719827}).  We mention three other approaches to Diophantine approximation on Carnot groups.

Hersonsky--Paulin approach Diophantine approximation on $\Heis^n$ from the point of view of complex hyperbolic geometry. Indeed, modulo Remark \ref{remark:HP}, in  \cite{MR1919402} they derive Theorem \ref{thm:SiegelIntro} (which states that $\alpha(\Sieg^n)=1$) from their study  \cite{MR1844559, MR1863852} of geodesics in cusped complex-hyperbolic manifolds, and are able to find Siegel Diophantine approximates for  \emph{every} point of $\Heis^n$ for $\alpha\leq 1$, which we are only able to do for $\Heis^1$. Our more direct approach is applicable to both the Carnot and Siegel models, and can be adapted to counting approximates (Theorem \ref{thm:IntroApproximates}) or prescribing denominators (Question \ref{question:denominators}). We intend to further explore the connection to complex hyperbolic geometry in an upcoming study.

\begin{remark}
\label{remark:HP}
It should be noted that \cite{MR1919402} is somewhat inconsistent about its definition of the height used for Diophantine approximation: while the complex-hyperbolic proofs employ the Siegel height $\height_{\Sieg^n}$, its definition is incorrectly stated using Carnot coordinates. 
\end{remark}

Another approach to Diophantine approximation on the Heisenberg group comes from the standard representation of the Heisenberg group $\Heis^1$ as the group of unipotent upper-triangular matrices in $SL(3,\R)$ (and analogously for all $\Heis^n$; see \cite{LV}). One can use this embedding to define rational points and a height function, but we are unable to say anything about its Diophantine properties.

A rather different approach is taken by Aka, Breuillard, Rosenzweig, and de Saxc{\'e} \cite{MR3357181,MR2586348}, who discuss a notion of \emph{Diophantine subgroups} of nilpotent groups $G$ (including Carnot groups). Given a finite set $S\subset G$, a number $N>0$, and a sequence $s_1, \ldots, s_N\in S\cup S^{-1}\cup \{id\}$, they ask how close  a non-identity product $s_1*s_2*\cdots*s_N$ can be to $id$, in terms of $N$. They define the group generated by $S$ to be Diophantine if it satisfies a certain condition on such distances. We are not aware of a connection between Diophantine subgroups and our notions of Diophantine approximation. In particular, iterating an element $g$ of a non-abelian Carnot group $G$ is not the same as dilating it: for a generic $g\in G$, $\delta_2(g) \neq g*g$.

\subsection{Questions}
 In this paper, we work with the Picard modular group $\Gamma=U(2,1;\Z[\ii])$ as the standard lattice in $\Hyp^2_\C$, and the group $\Sieg^1(\Z)$ as the standard lattice in the Heisenberg group. Many papers have studied Diophantine approximation by rationals of fields $\Q(\sqrt{-d})$ for positive, nonsquare integers $d$ \cite{Nakada,NW,Asmus3,Vulakh3s}.  While parts of our theory generalize when one replaces $\Z[\ii]$ above by $\Z[\sqrt{-d}]$, some parts require $\Z[\sqrt{-d}]$ to be a Euclidean domain, and others further require that there exists a fundamental domain for $\Sieg^n(\Z)$ that fits properly inside the unit ball in $\Sieg^n$. It would be interesting to see how the results change for other choices of lattice in $\Hyp^n_\C$ and $\Sieg^n$ and $\Carnot^n$, or for quaternionic hyperbolic space. In particular, we ask:
 
\begin{question}
What is the  Diophantine exponent of the boundary of quaternionic hyperbolic space?
\end{question}

It would be interesting to have more information on the set of badly approximable points in spaces $X$, including the Carnot groups and $\Sieg^n$ studied here. It appears to be an extremely difficult problem to identify all subsets $Y\subset X$ such that $\alpha(Y,X)\neq \alpha(X)$. A more tractable problem may be:
\begin{question}
Is the set of badly approximable points in $\Sieg^n$, $\Carnot^n$, or  $\C^n$ ($n\geq 2$) path-connected? Note that $\BA_{\R^n}$ is not path-connected since there exist rational hyperplanes in $\R^n$ which contain no badly-approximable points. 
\end{question}

In Theorem \ref{thm:IntroApproximates}, we provide a precise asymptotic for Diophantine approximation in Carnot groups. Our proof of the Diophantine exponent for the Siegel model of the Heisenberg group follows an analogous argument. One expects to be able to strengthen it.
\begin{question}
\label{q1}
Can one provide a precise asymptotic for Diophantine approximation in the Heisenberg group with respect to the Siegel height $\height_{\Sieg^n}$, analogously to Theorem \ref{thm:IntroApproximates}?
\end{question}

Lastly, it should be possible to adapt the proofs of Theorem \ref{thm:CarnotDiophantine} and Theorem  \ref{thm:SiegelIntro} to constrained choices of rationals. 
\begin{question}
\label{q2}
\label{question:denominators}
Does every point of $\Sieg^1$ admit infinitely many rational approximates with Gaussian prime denominators?
\end{question}

See \S \ref{sec:moredetailsonquestions} for more details on Questions \ref{q1} and \ref{q2}.

\subsection{Structure of the paper}

We prove Theorems \ref{thm:CarnotDiophantine} and \ref{thm:IntroApproximates} concerning Diophantine approximation in  Carnot groups in \S \ref{sec:Carnot}. We then focus on the first Heisenberg group in its Siegel model in \S \ref{sec:MT1},  using Heisenberg continued fractions to prove the trichotomy Theorem \ref{thm:trichotomy} and the case $n=1$ of the Siegel Diophantine approximation Theorem \ref{thm:SiegelIntro}. We prove the general case of Theorem \ref{thm:SiegelIntro} in \S \ref{sec:Siegeln}. We discuss badly approximable points and Schmidt games, proving Theorem \ref{thm:GeometryIntro}, in \S \ref{sec:badlyapprox}.

\subsection{Asymptotic notations}

We will make frequent use of asymptotic notations throughout this paper. By $f(x)=O(g(x))$ or $f(x) \ll g(x)$ we mean that there exists a constant $C$ (called the implicit constant) such that $|f(x)|\le C|g(x)|$ for sufficiently large $x$. By $f(x) \asymp g(x)$, we mean that $f(x) \ll g(x)$ and $g(x) \ll f(x)$. By $f(x) = o(g(x))$, we mean that $\lim_{x\to \infty} f(x)/g(x)= 0$. By $f(x) \sim g(x)$, we mean that $f(x) = g(x)(1+o(1))$ or, equivalently, that $\lim_{x\to \infty} f(x)/g(x) = 1$.

\section{Diophantine approximation in Carnot groups}
\label{sec:Carnot}
The goal of this section is to prove Theorem \ref{thm:CarnotDiophantine}, providing the Diophantine exponent for rational Carnot groups $G$ and in particular the Carnot model of the Heisenberg group.

\subsection{Carnot groups}
\label{sec:CarnotDefi}
A connected and simply connected Lie group $G$ with Lie algebra $\mathfrak g$ is a \emph{Carnot group} if $\mathfrak g$ admits a splitting $\mathfrak g = \mathfrak g_1 \oplus \cdots \oplus \mathfrak g_s$ satisfying $[\mathfrak g_1, \mathfrak g_i]=\mathfrak g_{i+1}$, with $\mathfrak g_{s+1}:=\{0\}$. Both $\R^n$ and the Heisenberg group $\Heis^n$ are examples of Carnot groups, and the general theory is largely analogous.  We recall some structure theory for Carnot groups; see \cite{MR3267520, MR1421823, MR0369608} for details.

The exponential mapping $\exp: \mathfrak g \rightarrow G$ is a bijection for nilpotent groups. Taking $n_i = \dim \mathfrak g_i$, we can give $G$ natural coordinates in $\R^{n_1}\times \ldots \times \R^{n_s}$. By the Baker-Campbell-Hausdorff formula, the group law on $G$ is then of the form
\begin{equation}
\label{eq:BCH}
(x_1, \ldots, x_s)*(x'_1, \ldots, x'_s)=(x_1+x_1', x_2+x_2'+p_2, \ldots, x_s+x_s'+p_s),
\end{equation}
where $x_i \in \R^{n_i}$ and $p_i$ is a vector of polynomials in the coordinates of $x_1, \ldots, x_{i-1}$ and $x_1', \ldots, x_{i-1}'$. 

One says that $G$ is \emph{rational} if one can choose the above coordinates such that the polynomials $p_i$ have integer coefficients. In this case the subset $G(\Z)$ having integer coordinates forms a co-compact lattice. Indeed, by a result of Malcev it is exactly the rational Carnot groups that admit co-compact lattices.

Specifying an inner product on $\mathfrak{g}_1$ leads to a Carnot path metric on $G$, which is in general difficult to compute. We will instead work with bi-Lipschitz equivalent metrics: the gauge metric \ref{defi:CarnotHeisenberg} on $\Heis^n$ and weighted infinity metrics on abstract Carnot groups.

Given positive weights $\Lambda = \{\lambda_1, \ldots, \lambda_s\}$, the \emph{$\Lambda$-weighted infinity (pseudo-) metric} on $G$ is defined by
\(
&d_{\infty,\Lambda}(g,g')=\Norm{g^{-1}*g'}_{\infty, \Lambda},
&\Norm{(x_1, \ldots, x_s)}_{\infty,\Lambda} = \max_{\stackrel{i=1,\ldots, s}{x\in x_i}}\norm{\lambda_i x}^{1/i}.
\)
For an appropriate choice of $\Lambda$, $d_{\infty,\Lambda}$ is in fact a metric (in particular, it satisfies the triangle inequality; see \cite{MR3267520,MR0369608}). We will assume such a choice of $\Lambda$ has been made and omit the subscript. Note that we may always assume $\lambda_1=1$, and that for $\Heis^n$ one may take $\lambda_2=\frac{1}{2n}$.

Setting $n=\sum_i n_i$ and $Q=\sum_i i\cdot n_i$, the space $(G,d_{\infty,\Lambda})$ has topological dimension $n$ but Hausdorff dimension $Q$. Furthermore, Lebesgue measure (which we work with) corresponds up to rescaling to both  the Hausdorff $Q$-measure and Haar measure.

For each $r>0$, the mapping $\delta_r(x_1, x_2, \ldots, x_s) = (rx_1, r^2x_2, \ldots, r^s x_s)$ dilates $d_\infty$ (as well as the gauge and Carnot metrics) by a factor of $r$ and satisfies $\delta_r (g* g')=\delta_r g * \delta_r g'$ for all $g,g'\in G$ (it is easy to see that $\delta_r$ commutes with the Lie bracket and the exponential map). It dilates volume by  $r^Q$.

The unit ball around the origin has volume $B_\Lambda:=2^n \prod_{i=1}^s {\lambda_i}^{-n_i}$ and form 
\[
[-\lambda_1^{-1},\lambda_1^{-1}]^{n_1}\times [-\lambda_2^{-1}, \lambda_2^{-1}]^{n_2} \times \dots \times [-\lambda_s^{-1}, \lambda_s^{-1}]^{n_s}.
\]
More generally, a ball of radius $r$ has measure $ B_\Lambda\cdot r^Q$.

\subsection{Upper bound for $\alpha(G)$: a Borel-Cantelli argument}
\label{sec:CarnotBorelCantelli}
\begin{thm}[Carnot Diophantine exponent upper bound]
\label{thm:GeometricUpperBound} 
Let $G$ be a rational Carnot group of Hausdorff dimension $Q$ equipped with a Carnot metric $d$, and let $\epsilon,C>0$. Then the set $ E$  of points $g\in G$ admitting infintely many solutions to the Diophantine equation
\begin{equation}
\label{eq:upperepsilon}
d(g, \delta_q^{-1}p)<Cq^{-\frac{Q+1+\epsilon}{Q}},  \text{ for }q\in \mathbb{N},\  p\in G(\Z),
\end{equation}
has measure zero.
\end{thm}

\begin{proof} The statement is invariant under bi-Lipschitz changes of metric, so we will instead work with a weighted infinity metric $d_{\infty, \Lambda}$, which we denote by $d$. We set $\alpha = (Q+1+\epsilon)/Q$ and denote by $n$ the topological dimension of $G$.

Since $d(g,\delta_q^{-1} p) = d(\gamma * g, \delta_q^{-1}(\delta_q \gamma * p))$ for any $\gamma \in G$, it suffices to prove that the set $E\cap X$ has measure zero, where $ X= [0,1)^n$. This set acts as a fundamental domain for $G(\Z)$ acting on $G$. That is, the sets $\gamma*X$, $\gamma\in G(\Z)$, cover the entire space $G$ with no overlaps. 

Let $q\in \N$ and let $E_q$ be the set of points $g\in X$ admitting a solution to \eqref{eq:upperepsilon} or, equivalently,
\[\label{eq:borelcantellifirsteq}
d(\delta_q g, p)<Cq^{-(1+\epsilon)/Q}.
\]
We note that the size of $E_q$ is at most the number of integer points $p\in G(\Z)$ satisfying \eqref{eq:borelcantellifirsteq} times the volume  of a ball of radius $Cq^{-\alpha}$, which will be at most  $O(q^{-Q-1-\epsilon})$.

To count the number of $p\in G(\Z)$ satisfying \eqref{eq:borelcantellifirsteq}, we note first that there exists some $r>0$ such that $X$ is contained in $B_r(0)$, the ball of radius $r$ around the origin. We claim that if $q$ is large enough, any $p$ satisfying \eqref{eq:borelcantellifirsteq} for some $g\in X$ must be in the ball $B_{2qr}(0)$. This follows from the triangle inequality since
\(
d(0,p) \le d(0,\delta_q g) + d(\delta_q g, p) \le qr + Cq^{-(1+\epsilon)/Q}
\) 
and this is less than $2qr$ provided $q$ is large enough.

Since $d$ is a weighted infinity metric, the ball $B_{2qr}(0)$ is a box containing $O(q^Q)$ integer points. 
Thus we see that $|E_q| =O(q^{-1-\epsilon})$.

Note now that $E\cap X \subset \cup_{q=N}^\infty E_q$ for each $N>0$. As $N$ grows, the volume of the latter set tends to zero, since it is bounded above by a constant times  the tail of the convergent series $\sum_{q=1}^\infty \norm{q}^{-1-\epsilon}$.
\end{proof}

\subsection{Lower bound  for $\alpha(G)$ and counting solutions}
\label{sec:CarnotLowerBound}

The goal of this section is to prove:
\begin{thm}
\label{thm:GeometricLowerBound}
Let $G$ be a rational Carnot group of Hausdorff dimension $Q\ge 3$ with a weighted infinity metric $d=d_{\infty, \Lambda}$. For a constant $C>0$, let $A_C(g,N)$ denote the number of solutions to
\(
d (g, \delta_q^{-1} p) \le C q^{-\frac{Q+1}{Q}}, \quad p\in G(\mathbb{Z}), \ q\in \mathbb{N},\  1\le q \le N.
\)

Let $\epsilon>0$ be fixed. Then for almost all $g\in G$ we have that 
\[\label{eq:DioLowerBoundAsymp}
A_C(g,N) = B_\Lambda \cdot  C^Q \log N + O\left( (\log N)^{1/2}(\log \log N)^{3/2+\epsilon}\right).
\]
\end{thm}

The exact asymptotic of \ref{thm:GeometricLowerBound} does not hold for an arbitrary metric bi-Lipschitz to $d_{\infty, \Lambda}$. Instead, adjusting the definition of $A_C(g,N)$ appropriately, we get that $A_C(g,N)\asymp \log N$ for almost all points $g$. In particular, the estimate $\alpha(G)\geq (Q+1)/Q$ holds for any metric bi-Lipschitz to $d_{\infty, \Lambda}$, such as the Carnot metric or the gauge metric for $\Heis^n$.

Note also that there is no restriction in Theorem \ref{thm:GeometricLowerBound} that the points $\delta_q^{-1} p$ be in ``lowest terms" in the sense that $q$ does not divide all of the coordinates of $p$. Thus the same rational point may be counted several times in this asymptotic.

Our primary tool in proving Theorem \ref{thm:GeometricLowerBound} is the following lemma.

\begin{lemma}[Lemma 1.5 in \cite{HarmanBook}]
\label{lemma:HarmanAlternate}
Let $X$ be a measure space with measure $\mu$ such that $0<\mu(X)<\infty$. Let $f_k(x)$, $k\in \mathbb{N}$, be a sequence of non-negative $\mu$-measurable functions and let $f_k$, $\phi_k$ for $k\in \mathbb{N}$ be sequences of real numbers such that $0\le f_k \le \phi_k$ for all $k$. Write $\Phi(N) = \sum_{k=1}^N \phi_k$, and suppose that $\Phi(N) \to \infty$ as $N\to \infty$. Suppose that for arbitrary integers $1\le m \le n$ we have
\[\label{eq:IntAsympEstimate}
\int_X \left( \sum_{m\le k \le n} (f_k(x) - f_k) \right)^2 \ d\mu \le K \sum_{m\le k \le n} \phi_k
\]
for an absolute constant K. Then for any given $\epsilon>0$, and for almost all $x$, we have, as $N \to \infty$,
\(
\sum_{k=1}^N f_k(x) = \sum_{k=1}^n f_k + O\left( \Phi^{1/2}(N) (\log (\Phi(N) +2))^{3/2+\epsilon} + \max_{1\le k \le N} f_k\right).
\)
\end{lemma}

In applications of Lemma \ref{lemma:HarmanAlternate}, we will take $\mu$ to be the Lebesgue measure, $\norm{X}=1$, $f_k(x)$ the indicator function of a set $\mathcal{E}_k$, and $f_k = \int_X f_k(x) \ dx =\norm{\mathcal{E}_k}$. In this case, the inequality \eqref{eq:IntAsympEstimate} simplifies to the following:
\(
\sum_{m\le k, k' \le n} \left( \norm{\mathcal{E}_k \cap \mathcal{E}_{k'}} - \norm{\mathcal{E}_k}\cdot \norm{\mathcal{E}_{k'}}\right) \le K \sum_{m\le k \le n} \phi_k.
\)

\begin{lemma}\label{lemma:LinearEquationCount}
Let $n,m\in \N$, let $k\in \R$, and let $A$ be a positive real number. The number of integer solutions $x,y$ to 
\(
|xn-ym+k| \le  A , \quad 0\le x < m,\quad  0 \le y < n
\) 
is at most $2A+\gcd(n,m)$.
\end{lemma}

\begin{proof}
By rewriting the equations, we see that this is equivalent to asking for 
\(
-A-k \le xn-ym \le A-k.
\)
Let us consider  a fixed $a\in [-A-k,A-k]$ and consider how many solutions there are to $xn-ym=a$.  Since, in this equation, $y$ is completely dependent on $x$ and since we only care about an upper bound on the number of solutions, it suffices to count how many $x$'s in the desired range $0 \le x < m$ have a solution.

The equation $xn-ym=a$ has no solutions unless $\gcd(n,m)|a$. So suppose this divisibility condition holds, and write $n'=n/\gcd(n,m)$, $m'=m/\gcd(n,m)$ and $a'=a/\gcd(n,m)$. Then we are really counting the number of solutions to $xn'-ym'=a'$ with $0\le x< m$, $0\le y < n$.  By taking this equation modulo $m'$ we see that $xn' \equiv a' \pmod{m'}$. Since $a'$ is fixed and $\gcd(n',m')=1$, this has a unique solution with $0\le x< m'$. And thus there are $m/m' = \gcd(n,m)$ solutions with $0\le x < m$. Thus if $\gcd(n,m)|a$ there are $\gcd(n,m)$ solutions, and if $\gcd(n,m)$ does not divide $a$ there are $0$ solutions. 

The number of multiples of $\gcd(n,m)$ in $[-A-k,A-k]$ is at most
\(
\frac{2A}{\gcd(n,m)}+1
\)
and so multiplying this by $\gcd(n,m)$ gives the desired bound.
\end{proof}

\begin{proof}[Proof of Theorem \ref{thm:GeometricLowerBound}]
We shall prove the theorem in the specific case where $G=\Heis^1$, as the general case is an extension of this. Note that here $(Q+1)/Q=5/4$. We may assume throughout this proof that we are taking $q$ to be sufficiently large, namely larger than some $N_0$, since ignoring all smaller $q$  changes $A_C(g,N)$ by at most a constant.

Here we take a weighted infinity metric that is given by
\(
\|(x,y,t)\|_\infty = \max\{ |x|, |y|, |\lambda t|^{1/2}\}.
\)
By our earlier calculation, the volume of a ball of radius $r$ is $B_\Lambda \cdot r^4 = 2^3 \lambda^{-1} r^4$.

As in the proof of Theorem \ref{thm:GeometricUpperBound}, consider the box $X=[0,1)^3$ seen as a subset of $\Heis^1$. This set has measure $1$ and acts as a fundamental domain for $\Heis^1(\Z) \backslash \Heis^1$. It suffices to prove that the desired asymptotic \eqref{eq:DioLowerBoundAsymp} holds for almost all $h \in X$, as the same proof will hold for almost all $h\in\gamma*X$ for any $\gamma\in \Heis^1(\Z)$. 

In order for $d_\infty(h, \delta_q^{-1} p)\le Cq^{-5/4}$ we must have that $h$ is in a ball of radius $Cq^{-5/4}$ centered at $\delta_q^{-1} p$. 
So we will apply Lemma \ref{lemma:HarmanAlternate} with $\mathcal{E}_q$ denoting the set of balls of radius $Cq^{-5/4}$ around rational points $\delta_q^{-1} p$, $p\in \Heis^1(\Z)$, intersected with $X$. To simplify notation, let us use $B(p,q)$ to denote a ball of radius $Cq^{-5/4}$ around $\delta_q^{-1} p$. Thus,
\(
\mathcal{E}_q = \bigcup_{p\in \Heis^1(\Z)} B(p,q) \cap X.
\)
(Here and for the remainder of the proof, we will always assume that $p$ or $p'$ refer to points in $\Heis^1(\Z)$.)

To apply Lemma \ref{lemma:HarmanAlternate}, we first need a bound on $|\mathcal{E}_q|$. 

We claim that in measuring $|\mathcal{E}_q|$ we may instead assume that $\mathcal{E}_q$ consists only of balls $B(p,q)$ with $\delta_q^{-1} p\in X$, now with no intersection of the ball with $X$.

In one direction, consider a ball $B(p,q)$ that intersects but is not a subset of $X$, so that it contains a point $h \not\in X$. Since $X$ is a fundamental domain for $\Heis^1(Z)$ acting on $\Heis^1$, there exists a unique $\gamma\in \Heis^1(\Z)$ such that $\gamma *h \in X$. In addition we have that $\gamma*B(p,q) = B(\delta_q \gamma * p,q)$ and $\delta_q \gamma* p\in \Heis^1(\Z)$, thus $\gamma *h$ belongs to some set $B(p',q)\cap X$. Thus each point in the alternate definition of $\mathcal{E}_q$ corresponds to a unique point in the original definition of $\mathcal{E}_q$. 

In the other direction, consider a point $h$ that belongs to a set $B(p,q)\cap X$ with $\delta_q^{-1} p\not \in X$. Since $X$ is a fundamental domain, there exists a unique $\gamma\in \Heis^1(\Z)$ such that $\gamma * \delta_q^{-1} p = \delta_q^{-1} (\delta_q \gamma * p) \in X$. Since $\gamma$ cannot be the identity element, we must have now that $\gamma *h $ belongs to a ball $B(p',q)$ such that $\delta_q^{-1} p'\in X$. Thus each point in the original definition of $\mathcal{E}_q$ corresponds to a unique point in the alternate definition of $\mathcal{E}_q$, and so there is a bijection between the points. Moreover, since this bijection is a piecewise translation, their measures must be the same.

Therefore, for the purposes of measuring the size of $\mathcal{E}_q$, we may assume it consists of the balls $B(p,q)$, $\delta_q^{-1} p\in X$. Since the corresponding $p$ must live in $[0,q)^2\times [0,q^2)$, we see that there are $q^4$ such balls. Let $\epsilon>0$ denote the minimum distance from the origin to any other point in $\Heis^1(\Z)$. So, if $p\neq p'$ and both are in $\Heis^1(\Z)$, then
\[
d_\infty (\delta_q^{-1} p, \delta_q^{-1} p') = \frac{1}{q} d_{\infty} (p,p') \ge \frac{\epsilon}{q}.
\]
But if the balls $B(p,q)$ and $B(p',q)$ overlap, then their centers must be within $2Cq^{-5/4}$ (twice the radius) of each other.
Therefore, provided $q$ is sufficiently large---and as noted at the start of this proof, we may assume that it is---each of the balls composing $\mathcal{E}_q$ are disjoint. Since each of them has size $B_\Lambda (Cq^{-5/4})^4 = B_\Lambda C^4 q^{-5}$, we obtain $\norm{\mathcal{E}_q} = B_\Lambda C^4 q^{-1}$.

Next we need to estimate the size of $\norm{\mathcal{E}_q \cap \mathcal{E}_{q'}}$. We may assume, without loss of generality that $q\le q'$. Also, if $q=q'$, then we already know the size of this set to be $B_\Lambda C^4 q^{-1}$, so we further assume that $q< q'$. 

By a similar argument as above, when measuring the size of $\norm{\mathcal{E}_q \cap \mathcal{E}_{q'}}$, we may assume that $\mathcal{E}_q$ consists of the balls $B(p,q)$ for $\delta_q^{-1}p \in X$, and we may assume that $\mathcal{E}_{q'}$ consists balls $B(p',q')$ around \emph{all} rational points $\delta_{q'}^{-1}p'$. We may make this large assumption on $\mathcal{E}_{q'}$ since the only parts about it we care about are those that intersect $\mathcal{E}_q$. We are continuing in our assumption that $q$ will be sufficiently large so that two distinct balls $B(p,q), B(p',q)$ from the set $\mathcal{E}_q$ do not overlap one another, and similarly for $q'$ and $\mathcal{E}_{q'}$.

Consider two balls $B(p,q)$ and $B(p',q')$ that intersect one another. Let us write out the centers of these intersecting balls, coordinate-wise, as $\delta_q^{-1}(a,b,c)$ and $\delta_{q'}^{-1}(a',b',c')$. Note that we must have $a,b\in [0,q)$, $c\in [0,q^2)$, $a',b'\in [0,q')$, and $c'\in [0,{q'}^2)$. Then we have that
\[\label{eq:CarnotProofEq3}
d_\infty \left(\delta_q^{-1}(a,b,c),\delta_{q'}^{-1}(a',b',c')\right) \le C(q^{-5/4}+{q'}^{-5/4}),
\]
Here the right inequality comes from the fact that if there is any overlap, then the distance between the centers must be at most the sum of the radii. To bound the number of intersecting balls, we must simply bound the number of solutions to \eqref{eq:CarnotProofEq3}.

  To do this, we take \eqref{eq:CarnotProofEq3} and dilate by $qq'$ to get 
\(
 d_\infty \left( (q'a,q'b,{q'}^2c),(qa',qb',q^2c')\right) \le C(q'q^{-1/4}+q{q'}^{-1/4}).
\)
(In a general Carnot group $G$, the $1/4$'s would be replaced by $1/Q$'s.)
Now we make use of the fact that  
\(
d_\infty ((x,y,t),(x',y',t'))=\max\{|x-x'|,|y-y'|,|\lambda(t-t'-2(xy'-x'y))|^{1/2}\}.
\)
This gives that
\(
& \max\{ |q'a-qa'|,|q'b-qb'|,|\lambda({q'}^2c-q^2c'-2qq'(ab'+2ba'))|^{1/2} \} \\&\qquad \le C(q'q^{-1/4}+q{q'}^{-1/4}).
\)

By Lemma \ref{lemma:LinearEquationCount}, we know that the number of possible solutions $(a,b,a',b')$ in the desired range to
\(
|q'a-qa'|\le C(q'q^{-1/4}+q{q'}^{-1/4})
\)
and
\(
|q'b-qb'|\le C(q'q^{-1/4}+q{q'}^{-1/4}).
\)
is at most $(2C (q'q^{-1/4}+q{q'}^{-1/4})+\gcd(q,q'))^2$. And for a fixed set $(a,a',b,b')$ we have that the number of solutions $(c,c')$ in the desired range to 
\(
 |{q'}^2c-q^2c'-2qq'(ab'+2ba')| \le  \lambda^{-1} C^2(q'q^{-1/4}+q{q'}^{-1/4})^2
\)
is at most $2 \lambda^{-1} C^2(q'q^{-1/4}+q{q'}^{-1/4})^2+\gcd(q,{q'})^2$, since $\gcd(q^2,{q'}^2) =\gcd(q,q')^2$. Now we want to multiply these bounds together. It is clear that $q'q^{-1/4}$ is always larger than $q{q'}^{-1/4}$; however, if $q'q^{-1/4}$ is smaller than $\gcd(q,q')$, then we just bound this product by \[ \label{eq:CarnotProofEq0} O(\gcd(q,q')^4),\] and if $q'q^{-1/4}$ is larger than $\gcd(q,q')$, we may bound the product by 
\[\label{eq:CarnotProofEq1}
B_\Lambda C^4 \frac{{q'}^4}{q}\left( 1+ O\left( \frac{q^{5/4}}{{q'}^{5/4}}\right) +O\left( \frac{q^{1/4} \cdot \gcd(q,q')}{q'} \right)\right).
\]
Here we have implicitly used the fact that $(1+x)^n = 1+O(x)$ provided $n$ is fixed and $x$ is bounded. Thus the total number of intersecting balls is at most the sum of \eqref{eq:CarnotProofEq0} and \eqref{eq:CarnotProofEq1}.

(In a general Carnot group $G$, we see that each time we apply Lemma \ref{lemma:LinearEquationCount}  to bound the number of solutions corresponding to integer points in some $x_i$, we get a constant factor of $2\lambda_i^{-1} C^i$. So for $G$ instead of getting $B_\Lambda C^4$ in \eqref{eq:CarnotProofEq1}, we will get $B_\Lambda C^Q$. There are also changes to the exponents, but we will comment on that more later.)

The size of the intersection between any two such balls is at most the volume of the smaller ball, which is $B_\Lambda C^4 {q'}^{-5}$. Multiplying this volume by the sum of \eqref{eq:CarnotProofEq0} and \eqref{eq:CarnotProofEq1}, we have that if $q< q'$, then
\(
\norm{\mathcal{E}_q \cap \mathcal{E}_{q'}} &\le \frac{B_\Lambda C^4}{q'} \cdot \frac{B_\Lambda C^4}{q}  +O\left( \frac{q^{1/4}}{{q'}^{9/4}}\right) + O\left( \frac{\gcd(q,q')}{q^{3/4} {q'}^{2}} \right) + O\left( \frac{\gcd(q,q')^4}{{q'}^5}\right)\\
&= \norm{\mathcal{E}_q}\cdot \norm{\mathcal{E}_{q'}}  +O\left( \frac{q^{1/4}}{{q'}^{9/4}}\right) + O\left( \frac{\gcd(q,q')}{q^{3/4} {q'}^{2}} \right) + O\left( \frac{\gcd(q,q')^4}{{q'}^5}\right).
\)

(In a general Carnot group $G$, we would have a similar bound, but with big-Oh terms that look like
\(
O\left( \frac{q^{1/Q}}{{q'}^{2+(1/Q)}}\right) + O\left( \frac{\gcd(q,q')}{q^{1-(1/Q)}{q'}^2}\right)+O\left( \frac{\gcd(q,q')^Q}{{q'}^{Q+1}}\right),
\)
and all the bounds we use below will apply just as well when these big-Oh terms are used.)

So for any integers $ m \le  n$ with $m$ sufficiently large, we have the following
\(
&\sum_{m\le q,q' \le n}\left(\norm{\mathcal{E}_q \cap \mathcal{E}_{q'}} - \norm{\mathcal{E}_q}\cdot \norm{\mathcal{E}_{q'}} \right)\\
 &\qquad= 2\sum_{m\le q < q' \le n}\left(\norm{\mathcal{E}_q \cap \mathcal{E}_{q'}} - \norm{\mathcal{E}_q}\cdot \norm{\mathcal{E}_{q'}} \right)+ \sum_{m\le q \le n}\left( \norm{\mathcal{E}_q } - \norm{\mathcal{E}_q }^2  \right)\\
&\qquad= 2\sum_{m\le q < q' \le n} \left( O\left( \frac{q^{1/4}}{{q'}^{9/4}}\right) +O\left( \frac{\gcd(q,q')}{q^{3/4} {q'}^{2}} \right)+ O\left( \frac{\gcd(q,q')^4}{{q'}^5}\right) \right) \\ &\qquad \qquad+ \sum_{m\le q \le n} O\left( \frac{1}{q}\right).
\)
Now we examine some of these sums individually.

First we note that $\sum_{1\le n \le x} n^\alpha \le \int_1^{x+1} t^\alpha \ dt = O(x^{\alpha+1})$ for any $\alpha>-1$. Therefore, we have
\(
\sum_{m\le q < q' \le n} \frac{q^{1/4}}{{q'}^{9/4}} &= \sum_{m < q' \le n} \frac{1}{{q'}^{9/4}} \left( \sum_{m \le q < q'} q^{1/4} \right) \\
&= \sum_{m < q' \le n} \frac{1}{{q'}^{9/4}} \cdot O\left( {q'}^{5/4}\right) = \sum_{m < q' \le n} O\left( \frac{1}{q'}\right).
\)

Next, let $\phi(n)$ denote the standard Euler totient function which counts the number of positive integers up to $n$ that are relatively prime to $n$. It is a standard fact of number theory that for any divisor $k$ of $q'$, we have that the number of $q$ such that $1\le q\le q'$ and $\gcd(q,q')=k$ is $\phi(q'/k)$.
Thus,
\(
\sum_{m\le q < q' \le n}\frac{\gcd(q,q')^4}{{q'}^5} &\le \sum_{m < q' \le n} \frac{1}{{q'}^5} \left( \sum_{1\le q \le q'} \gcd(q,q')^4\right)\\
& = \sum_{m < q' \le n} \frac{1}{{q'}^5}  \left( \sum_{k | q'} k^4 \cdot \phi(q'/k)\right).
\)
In this new internal sum we replace $k$ with $q'/k$ and note that $\phi(n)\le n$ for all integers $n$, so that
\(
\sum_{m\le q < q' \le n}\frac{\gcd(q,q')^4}{{q'}^5} &\le \sum_{m < q' \le n} \frac{1}{q'} \left( \sum_{k | q'} k^{-3} \right)\\
&=  \sum_{m < q' \le n} \frac{1}{q'} \cdot O(1)= \sum_{m < q' \le n} O\left(\frac{1}{q'} \right).
\)
(In the general case, it is here that we use our assumption that $Q\ge 3$, as otherwise, the sum $\sum_{k|q'} k^{-Q+1}$ will not be bounded.)

For the final sum, we will again consider $\gcd(q,q')=k$, so that $q=kd$ for some positive integer $d$. Technically, $d$ should be relatively prime to $q'$ but we may add in all possible $d$'s that are less than or equal to $q'/k$ (since $q<q'$) as this only increases the size of the sum. Thus,
\(
\sum_{m\le q < q' \le n} \frac{\gcd(q,q')}{q^{3/4} {q'}^{2}} &\le \sum_{m \le q' \le n} \frac{1}{{q'}^2} \sum_{k| q'} k^{1/4} \sum_{d\le q'/k} d^{-3/4}\\
&= \sum_{m\le q' \le n} \frac{1}{{q'}^2} \sum_{k| q'} k^{1/4} \cdot O\left( \left( \frac{q'}{k}\right)^{1/4}\right)\\
&= \sum_{m\le q'\le n} \frac{1}{{q'}^{7/4}} O\left( \sum_{k| q'} 1\right).
\)
This internal sum is the number of divisors of $q'$, often denoted $d(q')$. It is a well-known number-theoretic fact that $d(m) =O(m^{\epsilon})$ for any fixed $\epsilon>0$. Thus, in particular, we see that each summand may be bounded by $O({q'}^{-1})$, just as with all the others.

Therefore, we have that
\(
&\sum_{m\le q,q' \le n}\left( \norm{\mathcal{E}_q \cap \mathcal{E}_{q'}} - \norm{\mathcal{E}_q} \cdot \norm{\mathcal{E}_{q'}} \right) =\sum_{ m \le q \le n} O\left( \frac{1}{q}\right).
\)
Taking $\phi_k$ to be a suitably large constant times $k^{-1}$, we see that all the conditions of Lemma \ref{lemma:HarmanAlternate} hold. Noting that $\sum_{1\le n \le x} n^{-1} \sim \log x$ completes the proof.
\end{proof}

\section{Hyperbolic geometry and continued fractions}\label{sec:MT1}

In this section, we use Heisenberg continued fractions to prove the case $n=1$ of Theorem \ref{thm:SiegelIntro}, and the trichotomy Theorem \ref{thm:trichotomy}. 

\subsection{Complex hyperbolic space}
\label{sec:coredefinitions}

We recall some notions in complex hyperbolic geometry, see e.g.\ \cite{MR2753829, MR1695450, MR1189043,MR2987619} for more details.
The complex hyperbolic plane $\Hyp^2_\C$ is the unit ball in $\C^2 \subset \CP^2$, with a metric invariant under linear fractional transformations $U(2,1)$ preserving the ball. No linear fractional transformation transforms this ball into a half-space; instead, sending a boundary point off to infinity provides the Siegel model of $\Hyp^2_\C$. Namely, we view $\Hyp^2_\C$ as the space
$$\Hyp^2_\C = \{ (u,v) \in \C^2 \st 2 \Re(v) > \norm{u}^2\},$$
whose boundary is the one-point compactification of the Siegel space $\Sieg^1 = \{ (u,v) \in \C^2 \st 2 \Re(v) = \norm{u}^2\}$.

The parabolic elements of $U(2,1)$ preserving the point at infinity act simply transitively on $\partial \Hyp^2_\C=\Sieg^1$, giving it a group structure. Furthermore, from the theory of hyperbolic spaces one obtains a natural equivalence class of \emph{parabolic visual metrics} on $\Sieg^1$. The left-invariant metric given by the Koranyi norm $\Norm{(u,v)}=\norm{v}^{1/2}$ is the simplest of these.

The Picard modular group $\Gamma=U(2,1;\Z[\ii])$ is the direct analogue of the classical modular group. It is generated \cite{MR2753829} by the Koranyi inversion $\iota: (u,v)\mapsto (-u/v, 1/v)$  analogous to the classical $x\mapsto 1/x$, the rotation $(u,v)\mapsto (\ii u,v)$, and parabolic transformations acting simply transitively on the Gaussian integer points $\Sieg^1(\Z)$ of $\Sieg^1$.

\subsection{Continued fractions and $\alpha(\Sieg^1)$}
\label{sec:CFalpha}

Let $\{\gamma_i\}$, $i \ge 0$, be a possibly finite sequence in $\Sieg^1(\Z)$. The \emph{Heisenberg continued fraction} with \emph{digits} $\{\gamma_i\}$ is given by
$$\mathbb K \{\gamma_i\} = \lim_{n\rightarrow\infty} \gamma_0 * \iota(\gamma_1 * \iota (\gamma_2 *\dots *\iota( \gamma_n))),$$
if this limit exists. We showed in \cite{LV} that every point $h\in \Sieg^1$ admits a continued fraction expansion, and characterized rational points in $\Sieg^1$ as exactly those points with a finite expansion.

The continued fraction digits of $h$ are found in \cite{LV} by repeatedly applying a generalized Gauss map $T$ to the point $h$. More explicitly, let $K$ be the Dirichlet region for $\Sieg^1(\Z)$ centered at the origin, i.e., $K$ is the set of points closer to $0$ than to any other point in $\Sieg^1(\Z)$ with some choice of boundary. Let $[\cdot]$ denote the nearest-integer map from $\Sieg^1$ to $\Sieg^1(\Z)$ characterized by the property $h\in [h]*K$. For $h\in K$, we then take $Th= [\iota(h)]^{-1}* \iota(h)$ if $h\neq 0$ and $T0=0$ otherwise. (Recall that the classical Gauss map is defined by $Tx = \frac{1}{x}-\lfloor \frac{1}{x}\rfloor$ for $x\in (0,1)$.) For each $i\ge 0$, set $T^ih=h_i = (u_i, v_i)$. For a point $h\in K$, the continued fraction digits are given by $\gamma_0=(0,0)$ and $\gamma_i=[\iota h_i]$. More generally, we let $\gamma_0=[h]$ and define the remaining digits by replacing $h$ with $[h]^{-1} *h $.

The convergents of a continued fraction can be canonically written as rational numbers, and we write for each $n$:
\(
\left( \frac{r_n}{q_n}, \frac{p_n}{q_n} \right)= \gamma_0 * \iota(\gamma_1 * \iota (\gamma_2 *\dots *\iota( \gamma_n))).
\)
It can be shown that the convergents are already in lowest terms, that is, there is no Gaussian prime $\pi$ that divides $r_n,p_n,q_n$ simultaneously.

Our past results on Heisenberg continued fractions immediately provide:
\begin{thm}\label{thm:sieg1critexp}
The  Diophantine exponent $\alpha(\Sieg^1)$ for the Siegel model of $\Heis^1$ is 1. 
\end{thm}
\begin{proof}
Consider the continued fraction expansion of an irrational point $h$ in $\Sieg^1$ (i.e.\ a point in $\Sieg^1\backslash \Sieg^1(\Q)$). The convergence rate of the approximates is given \cite{VDiophantine} by
\begin{equation}
\label{eq:boundeddistancesiegel}
d\left( \left( \frac{r_n}{q_n}, \frac{p_n}{q_n} \right), h\right) \asymp \frac{|v_{n+1}|^{1/2}}{|q_n|}.
\end{equation}
Thus, continued fractions convergents give infinitely many approximations approximations of $h$ with error on the order of $\norm{q}^{-1}$.

On the other hand, a Borel-Cantelli argument (cf.\ Theorem \ref{thm:GeometricUpperBound}) in \cite{VDiophantine} shows that the exponent cannot be improved.
\end{proof}

\begin{remark}
Because \eqref{eq:boundeddistancesiegel} holds for all points $h$, Theorem \ref{thm:sieg1critexp} provides a lower bound on the irrationality exponent for \emph{all} points in $(\Heis^1, \height_{\Sieg^1})$.
\end{remark}

We will also need to know that continued fractions give (essentially) best approximates:
\begin{prop}\label{prop:bestapproximation}
Let $h\in \Sieg^1$ be an irrational point with $n^{\text{th}}$ convergent $(r_n/q_n,p_n/q_n)$. There exists a constant $M>0$, independent of $h$, such that if $(R/Q,P/Q)$ is a different rational point in lowest terms with $|Q|< M |q_n|$, then 
\(
d\left( \left( \frac{R}{Q}, \frac{P}{Q}\right), h \right) > d\left( \left( \frac{r_n}{q_n}, \frac{p_n}{q_n} \right), h \right).
\)
\end{prop}

\subsection{Badly approximable points and bounded continued fraction digits}
\label{sec:BDBA}
We are now in position to prove that $\BA_{\Sieg^1}=\BD$.
\begin{thm}\label{thm:BAisBD}
Let $h$ be an irrational point of $\Sieg^1$. Then $h$ is badly approximable---that is, there exists a constant $C_h>0$ such that for \emph{all} rational points $(R/Q,P/Q)$ we have that 
\[\label{eq:BAisBD}
d\left( \left( \frac{R}{Q},\frac{P}{Q}\right), h\right) > \frac{C_h}{|Q|}
\]
--- if and only if the continued fraction digits of $h$ are bounded.
\end{thm}

\begin{proof}
The key idea is the relationship between the forward iterates $h_n$ and the continued fraction digits $\gamma_n$. We know that $\gamma_n = [\iota h_{n-1}]$. Since $\Norm{\iota h} =\Norm{h}^{-1}$, we have that $|\gamma_n|$ is bounded if and only if the sequence $h_n$ (and thus $v_n$) is bounded away from the origin. 

So in one direction the proof is immediate. If the CF digits of $h$ are unbounded, then $v_n$ will be arbitrary close to $0$, but then by \eqref{eq:boundeddistancesiegel}, we have that $h$ cannot be badly approximable since the convergents themselves will make \eqref{eq:BAisBD} false no matter how small $C_h$ is.

For the other direction we must be slightly more careful. Suppose that $h$ has bounded continued fraction digits. Then we know that $|v_n|$ is bounded away from $0$ and we know it is bounded from above by $1$, since $K$ is contained in the unit sphere, so that $|v_n|\asymp 1$. Thus since $|q_{n-1}| \asymp |v_{n-1} \cdot q_n|$ (proved in \cite{VDiophantine}), we have that $|q_{n-1}|\asymp |q_n|$.

Let $(R/Q,P/Q)$ be a rational point. Let $M$ be the constant in Proposition \ref{prop:bestapproximation}. Without loss of generality we may assume that $Q$ is large, so that there exists an $n$ with $M|q_{n-1}|\le |Q| <M |q_n|$. (The $n$ may not be unique since the norm of the $q_n$'s does not need to be monotone increasing.) Since $|q_{n-1}| \asymp |q_n|$ we also have that $|Q| \asymp |q_n|$.

Thus we have that
\(
d\left( \left( \frac{R}{Q},\frac{P}{Q}\right), h \right) &> d\left( \left( \frac{r_n}{q_n} , \frac{p_n}{q_n}\right), h\right) = \frac{\norm{v_{n+1}}^{1/2}}{\norm{q_n}} \gg \frac{1}{\norm{q_n}} \gg \frac{1}{\norm{Q}}.
\)
Since the implicit constant does not depend upon $Q$ in any way, we have shown that $h$ is badly approximable in this case.
\end{proof}

Combining with Theorem \ref{thm:BAisBD}, we will obtain:
\begin{cor}\label{thm:weirdergodic}
Almost all $h\in \Sieg^1$ have arbitrarily large continued fraction digits and arbitrarily small $h_i$.
\end{cor}

Such a result would normally be a simple consequence of the ergodicity of the continued fraction map, but this  is not known in the Heisenberg group.

\subsection{Badly approximable points and bounded geodesics}\label{sec:horospheres}
\label{sec:DisBA}

Our goal is now to show that the set of badly approximable points $\BA_{\Sieg^1}$ coincides with the Diophantine set $D(\Gamma)$.

Recall that $D(\Gamma)$ is defined as the set of endpoints of geodesic rays in $\Hyp^2_\C$ that stay in a bounded region of the cusped manifold $\Gamma\backslash \Hyp^2_\C$. That the cusp of $\Gamma\backslash \Hyp^2_\C$ lifts to a $\Gamma$-invariant collection of horoballs in $\Hyp^2_\C$, so $h$ is in $D(\Gamma)$ if and only if  some (equivalently any) geodesic ray ending in $h$ misses some $\Gamma$-invariant collection of horoballs.

Horospheres in $\Hyp^2_\C$ are analogous to those in $\Hyp^2_\R$. Namely, a horosphere with basepoint $\infty$ and (horo-)height $s>0$ is composed exactly of points $\{(a,b+s^2)\st 2\Re(b)=\norm{a}^2\}$ (see \cite{MR1189043} where horospherical coordinates were first introduced). In particular, any point $(u,v)\in \Hyp^2_\C$ is contained in a horosphere based at infinity with height
\begin{equation}
\label{eq:horoheight}
s^2 = \Re(v-\norm{u}^2/2).
\end{equation}
All other horospheres can be obtained by applying elements of $U(2,1)$ to a horosphere based at $\infty$. While a horosphere based at infinity has constant horoheight, while the horoheights of points in a horosphere based at a finite point vary up to some maximum value which we call the (horo-)height of the horosphere. A \emph{horoball} of height $s_0$ based at infinity is defined by the condition $s\geq s_0$. Horoballs elsewhere are defined by the $U(2,1)$ action.

The \emph{shadow} of a horoball in $\Hyp^2_\C$ is the set of points in $\Sieg^1$ occluded by the horosphere. That is, it is the collection of endpoints of vertical geodesic rays that intersect the horosphere. Here, a \emph{vertical} geodesic ray has the form $\xi(t) = (u, v+e^{-2t})$ for $t\in [a,\infty)$ and $(u,v)\in \Sieg^1$. The constant $2$ in the exponent is due to the standard normalization that makes $\Hyp^2_\C$ a CAT(-1) space. 

We now record some facts about horoballs.  The following is straightforward:
\begin{lemma}
\label{lemma:fliphorosphere}
Let $H$ be a horosphere based at $\infty$ at height $s$, and $\iota$ the Koranyi inversion. Then $\iota(H)$ is a horosphere based at 0 with height $1/s$.
\end{lemma}

\begin{lemma}
Let $B$ be a horoball based at a point $h\in \Sieg^1$, of height $s$. Then the shadow of $B$ is comparable to a ball of radius $s$ based at $h$.
\begin{proof}
Left translation by $h^{-1}$ acts by isometries on both $\Sieg^1$ and $\Hyp^2_\C$ while preserving horoheight and taking vertical geodesics to vertical geodesics. It therefore suffices to consider $h=0$. Likewise, the mapping $\delta_{s}^{-1}(u,v)=(s^{-1}u, s^{-2}v)$ reduces us to the case of a horoball of height $1$. Lastly, it is easy to show using the Koranyi inversion that the shadow of the horoball based at $h=0$ of height $1$ is both bounded and contains an open set around the point $h=0$. Thus, its shadow is contained in an open ball of some radius $R$ and contains a ball of radius $R^{-1}$.
\end{proof}
\end{lemma}

\begin{lemma}
\label{lemma:horopositions}
Let $H$ be a horoball of height $s$ based at a point $(u,v)\in \Sieg^1$. Then $\iota(H)$ is a horoball based at $(-u/v, 1/v)$ of height $s\Norm{(u,v)}^{-2}=s\norm{v}^{-1}$. 
\begin{proof}
Since $H$ is based at $(u,v)$, it contains the point $(u,v+s^2)$. Notice that if we invert $H$, translate it so that it is based at $0$ (but still has the same height), and invert it again, we get a horoball based at $\infty$. The point $\iota(\iota(u,v)^{-1}*\iota(u,v+s^2))$ will then be in this horoball, and we will be able to use \eqref{eq:horoheight} to calculate its height. We finish by applying Lemma \ref{lemma:fliphorosphere}.

We now compute (using the fact that $2\Re(v)=\norm{u}^2)$:
\begin{align*}
\iota( \iota(u,v)^{-1}*\iota(u,v+s^2)) &= \iota \left(\left(-\frac{u}{v}, \frac{1}{v}\right)^{-1}*\left(\frac{-u}{v+s^2}, \frac{1}{v+s^2}\right)\right)\\
&= \iota \left(
\frac{-u}{v+s^2}+\frac{u}{v}, \frac{1}{\overline v} + \frac{\overline u}{\overline v}\frac{-u}{v+s^2}+\frac{1}{v+s^2}\right)\\
&=\iota \left(\frac{us^2}{v(v+s^2)}, \frac{s^2+2 \Re(v)-\norm{u}^2}{\overline{v}(v+s^2)}\right)\\
&= \iota \left(\frac{us^2}{v(v+s^2)}, \frac{s^2}{\overline{v}(v+s^2)}\right)\\
&= \left ( \frac{-u \overline v}{v}, \frac{\overline v (v+s^2)}{s^2}\right).
\end{align*}
The last line gives us the coordinates of a point in a horosphere based at $\infty$. We compute its height using \eqref{eq:horoheight}:
$$\Re\left(\frac{\overline v (v+s^2)}{s^2}- \frac{1}{2}\norm{
\frac{u \overline v}{v}}^2\right)= \Re\left( \frac{\norm{v}^2}{s^2} + \overline v - \frac{1}{2}\norm{u}^2\right)
=\frac{\norm{v}^2}{s^2},$$

where the last equality follows from the condition $2\Re(v)=\norm{u}^2$. Applying Lemma \ref{lemma:fliphorosphere} completes the proof.
\end{proof}
\end{lemma}

We now fix a collection $\mathcal B$ of horoballs arising from the action of $\Gamma$ on $\Hyp^2_\C$. Because $\Gamma$ acts transitively on $\Sieg^1(\Q)\cup \{\infty\}$, we have that $\mathcal B$ contains a unique horoball based at each rational point. Furthermore, the height of the horoball $B_0\in \mathcal B$ based at $0$ completely determines $\mathcal B$. We now compute the heights of the remaining horoballs in $\mathcal B$:

\begin{lemma}
\label{lemma:horoheights}
Let $\mathcal B$ be as above, with the horoball $B_0$ based at the origin having horoheight $s_0$. If $(p/q, r/q)\in \Sieg$ is in lowest terms, then the horoball at $(p/q, r/q)$ has horoheight $s_0 \norm{q}^{-1}$.
\begin{proof}
We will approach the point $h:=(p/q, r/q)$ through a series of translations and inversions as prescribed by its continued fraction expansion. 

We use the continued fraction terminology of \S \ref{sec:CFalpha}. Applying left translation by an element of $\Sieg^1(\Z)$, we may assume without loss of generality that $h\in K$.  Denote the forward iterates $T^ih$ of $h$ by $T^i h = h_i=(u_i, v_i)$, including $(u_0, v_0)=(p/q, r/q)$. Theorem 3.10 of \cite{LV} guarantees that for some $n$ we have $(u_n, v_n)=(0,0)$, so that the associated continued fraction digit sequence $\{\gamma_1, \ldots, \gamma_n\}$ is finite. Since the Gauss map acts by a forward shift on the expansion, the continued fraction digits of $h_i=(u_i,v_i)$ are $\{\gamma_{i+1},\gamma_{i+2},\dots,\gamma_n\}$. 

Consider now a horoball $B_h$ based at $h=h_0=\iota (\gamma_1 * \iota (\gamma_2  \cdots \iota (\gamma_n *0)\cdots))$ with unknown horoheight $s$. Interpreting Heisenberg left multiplication and $\iota$ as elements of $\Gamma$ as in \cite{LV}, we have $B_h=\iota (\gamma_1 * \iota (\gamma_2  \cdots \iota (\gamma_n *B_0)\cdots))$.

Applying the Gauss map, we obtain the horoball $\gamma_1^{-1}*\iota(B_h)$ based at $h_1$. By Lemma \ref{lemma:horopositions}, the new horoball has height $s\Norm{h}^{-2}=s\norm{v_0}^{-1}$ (note that translation by $\gamma_1^{-1}$ does not affect horoheight). Applying the Gauss map again, we see a horoball of height $s\norm{v_0v_1}^{-1}$, and so on, until we see the horoball $B_0$ based at the origin. Since $B_0$ has height $s_0$, we have $s_0 = s\norm{v_0\cdots v_{n-1}}^{-1}$. Now, equation (3.4) of \cite{LV} in the case $(u_n, v_n)=(0,0)$ implies $|v_0\cdots v_{n-1}|^{-1}=|q|$ (the equivalence between our $q$ and the $q_n$ in the equation is due to $(p/q, r/q)$ being in lowest terms, so they can at most differ by a unit), so we have $s_0 = s\norm{q}$, as desired.
\end{proof}
\end{lemma}

We can now complete the proof of:
\begin{thm}
A point $h\in \Sieg^1$ is badly approximable if and only if it is in the Diophantine set $D(\Gamma)$.
\begin{proof}
Suppose $h \in D(\Gamma)$. Let $\xi$ be vertical geodesic ray in $\Hyp^2_\C$ that terminates at $h$. In the quotient $\Gamma \backslash \Hyp^2_\C$, $\xi$ stays in a bounded region, so does not enter some part of the cusp. Lifting back to $\Hyp^2_\C$, this means that $\xi$ avoids a $\Gamma$-invariant collection of horoballs $\mathcal B$ (whose image in the quotient is the avoided cuspidal region). Let $C$ be the horoheight of the horoball based at 0. By Lemma \ref{lemma:horoheights}, each horoball $B\in \mathcal B$ whose base is represented in reduced form as $(p/q,r/q)$ has horoheight $C \norm{q}^{-1}$, and has a shadow that is comparable to a ball of radius $C \norm{q}^{-1}$ around the basepoint. Since $\xi$ misses $B$, it is not contained in its shadow. We conclude $d(h, (p/q,r/q))>C' \norm{q}^{-1}$ for some $C'>0$  dependent only on $C$.
The converse is analogous.
\end{proof}
\end{thm}

\section{Siegel model $\Sieg^n$}
\label{sec:Siegeln}
The goal of this section is to prove Theorem \ref{thm:SiegelIntro} for all $n$. Recall that we proved the case $n=1$ using continued fractions in \S \ref{sec:CFalpha}, and that results analogous to Theorem \ref{thm:SiegelIntro} follow from \cite{MR1919402}, see \S \ref{sec:otherapproaches}.

The Borel-Cantelli argument (cf.\ Theorem \ref{thm:GeometricUpperBound}) for $\Sieg^1$ given in \cite{VDiophantine} generalizes directly to give $\alpha(\Sieg^n)\leq 1$. We now prove that $\alpha(\Sieg^n)\geq 1$  (note that the proof will take $n=1$ for simplicity):

\begin{thm}\label{thm:sieg1lowerbound}
Let $C>0$. For almost all $h\in \Sieg^n$, that there exist infinitely many rational points $(r_1/q,r_2/q,\dots,r_n/q,p/q)\in \Sieg^n(\Q)$ satisfying
\(
d\left( h, \left( \frac{r_1}{q},\frac{r_2}{q},\dots, \frac{r_n}{q},\frac{p}{q}\right)\right) \le \frac{C}{|q|}.
\)
\end{thm}

We recall that $(u_1,u_2,\dots,u_n,v)\in \Sieg^n$ if $|u_1|^2+|u_2|^2+\dots+|u_n|^2 -2\Re(v) =0$. Moreover, $\Sieg^n(\Q)= \Sieg^n \cap \Q[\ii]^{n+1}$.

The proof of Theorem \ref{thm:sieg1lowerbound} follows a similar scheme as that of Theorem \ref{thm:GeometricLowerBound}. Namely, we will use Lemma \ref{lemma:Harman} in place of Lemma \ref{lemma:HarmanAlternate} to show that infinitely many approximations are available to a positive-measure set of points $h\in \Sieg^n$, and then use Lemma \ref{lem:SiegZO}, a zero-one law, to get an almost-all statement.

\begin{remark}
The results in this section provide an alternate and more direct proof of a result of Hersonsky-Paulin \cite{MR1919402} (with the corrected notion of Diophantine approximation, see \ref{sec:otherapproaches}).
\end{remark}

We start by establishing a zero-one law for $\Sieg^n$ (note that similar results hold for Carnot groups):

\begin{lemma}
\label{lemma:ErgodicDenseActions}
Let $E\subset \Sieg^n$ be a measurable set invariant under a group $\Gamma$ of isometries of $\Sieg^n$ whose orbits are dense. Then either $E$ has zero measure or its complement has zero measure. 
\begin{proof}

Assume by way of contradiction that both $E$ and $E'=\Sieg^n \backslash E$ have positive Lebesgue measure in the geometric model (equivalently, positive Haar measure, or Hausdorff $Q$-measure). 

The (generalized) Lebesgue Density Theorem \cite{MR1616732} states that for any measurable set $A$, almost every point $x\in A$ is a \emph{point of density}. That is,
\[
\label{eq:pointOfDensity}
\lim_{r\rightarrow 0}  \frac{\norm{A\cap B_r(x) }}{\norm{B_r(x)}}=1 \text{ for a.e.\ }x\in A.
\]

Let $x, x'$ be points of density for $E\cap \Sieg^n$ and $E'\cap \Sieg^n$, respectively. Choose $r>0$ be such that $\norm{E\cap B_r(x)}/\norm{B_r(x)}>3/4$ and $\norm{E'\cap B_r(x')}/\norm{B_r(x')}>3/4$. By the density of orbits under $\Gamma$, we have that there exists some $\gamma\in\Gamma$ such that the measure of $(\gamma * B_r(x)) \triangle B_r(x')$ does not exceed $1/8$, and thus it follows that $\norm{E\cap B_r(x')}/\norm{B_r(x')}>1/2$, which contradicts $\norm{E'\cap B_r(x')}/\norm{B_r(x')}>3/4$.
\end{proof}
\end{lemma}

We note that a ball of radius $r$ in $\Sieg^n$ with respect to the gauge metric has volume $c r^{2n+2}$ where $c$ is the volume of a ball of radius $1$.

\begin{lemma}
\label{lemma:SiegelRationalInvariance}
Let $\alpha>0$ and let $E_\infty$ be the set of points $(u,v)\in \mathcal S$ that for some $C>0$ (possibly depending on the point) we have infinitely many solutions to 
\[\label{eq:SiegelRationalInvariance}d((u,v), (r/q, p/q))< C \norm{q}^{-\alpha}\]
for $r,p,q\in\Z[\ii]$. Then $E_\infty$ is invariant under rational Heisenberg translations.
\begin{proof} 
Suppose $(u,v)\in E_\infty$ and let $(r'/q', p'/q')\in \Sieg^1(\Q)$ be a rational point. We would like to show that $(r'/q',p'/q')*(u,v)\in E_\infty$.

Suppose $(r/q, p/q)$ is a solution to \eqref{eq:SiegelRationalInvariance} with constant $C$. By the left-invariance of the metric, we obtain a new rational point close to $(r'/q',p'/q')*(u,v)$:
\begin{align*}
C \norm{q}^{-\alpha} &> d((u,v), (r/q, p/q))\\ &= d((r'/q', p'/q')*(u,v), (r'/q', p'/q')*(r/q, p/q))
\end{align*}
This new rational point can be written as a fraction with $q\norm{q'}^2$ as the denominator:
\begin{align*}
(r'/q', p'/q')*(r/q, p/q) &= (r'/q'+r/q, p'/q'+p/q+(\overline{r'/q'})(r/q))\\
						 &=  \left ( \frac{r' \overline{q'} q+r\norm{q'}^2}{q\norm{q'}^{2}}, \frac{ p'\overline{q'}q+p+\overline{r'}q'r}{q\norm{q'}^{2}}
						 \right).
\end{align*}
We have thus found a solution to \eqref{eq:SiegelRationalInvariance} at the new point $(r'/q',p'/q')*(u,v)$, with the constant $C$ replaced by $C\norm{q'}^{2\alpha}$. If there are infinitely many solutions at $(u,v)$, they yield an infinite number of solutions at $(r'/q',p'/q')*(u,v)$, so $(r'/q',p'/q')*(u,v)\in E_\infty$, as desired.
\end{proof}
\end{lemma}

\begin{lemma}[Siegel zero-one law]\label{lem:SiegZO}
Let $\alpha\in \mathbb{R}$ be fixed. For $C>0$, let $E_C$ be the set of points $(u,v)$ satisfying
\(
d\left( (u,v), \left(\frac{r}{q}, \frac{p}{q}\right)\right) < C \norm{q}^{-\alpha}
\)
for infinitely many choices of $r,p,q\in \Z[\ii]$. Then either $E_C$ has full measure for all choices of $C$, or its complement does.
\begin{proof}

Consider the collection of sets $E_i$ defined as the set of points $(u,v)\in \Sieg^n$ satisfying
\begin{equation}
\label{eq:2cq}
d\left( (u,v), \left(\frac{r}{q}, \frac{p}{q}\right)\right) \leq  2^i C q^{-\alpha}
\end{equation}
for infinitely many choices of $ (r/q, p/q)\in \Sieg^n (\Q)$. Note that we have $E=E_0\subset E_1\subset E_2 \subset \cdots \subset \cup_i E_i =: E_\infty$. It is immediate by Lemma \ref{lemma:SiegelRationalInvariance} that $E_\infty$ is invariant under rational translations. 
Since rational translations are a dense set of isometries, either $E_\infty$ has zero measure or its complement does by Lemma \ref{lemma:ErgodicDenseActions}. If it has zero measure, then we are done by containment. If it has full measure, then we are done if we can show that $\norm{E_i\backslash E_0}=0$ for all $i\ge 1$. To simplify notation, we show this for $i=1$ and $C=1$.

Indeed, fix a positive integer $M$ and let $E^M$ be the set of points $(u,v)$ such that, for $i=1$, the formula \eqref{eq:2cq} has infinitely many solutions, while for $i=0$ all the solutions to \eqref{eq:2cq} satisfy $q<M$. Note that $\cup_M E^M = E_1\backslash E_0$. It suffices to show that for all $M$ we have $\norm{E^M}=0$.

Suppose $\norm{E^M}\neq 0$ and let $(u,v)\in E^M$ be a point of density, as in \eqref{eq:pointOfDensity}.  Let $p,q$ be a solution to \eqref{eq:2cq} with $i=1$. Because we have infinitely many solutions with $i=1$, we may assume $q>M$. Furthermore, because $h$ is a point of density we have
$$\norm{B_{2 q^{-\alpha}} \left(\frac{r}{q}, \frac{p}{q}\right)\cap E^M}= \norm{B_{2 q^{-\alpha}} \left(\frac{r}{q}, \frac{p}{q}\right)}(1+o(1)),$$
where the asymptotic holds as $q$ goes to infinity. 

Noting that $E^M$ is disjoint from $B_{q^{-\alpha}}(r/q,p/q)$, one computes
\begin{align*}
\norm{B_{2 q^{-\alpha}} \left(\frac{r}{q}, \frac{p}{q}\right)} &\geq \norm{B_{q^{-\alpha}} \left(\frac{r}{q}, \frac{p}{q}\right)} + \norm{B_{2 q^{-\alpha}} \left(\frac{r}{q}, \frac{p}{q}\right)\cap E^M}\\
&=2^{-2n-2} \norm{B_{2 q^{-\alpha}} \left(\frac{r}{q}, \frac{p}{q}\right)} + \norm{B_{2 q^{-\alpha}} \left(\frac{r}{q}, \frac{p}{q}\right)\cap E^M}\\
&= (1+2^{-2n-2}+o(1))\norm{B_{2 q^{-\alpha}} \left(\frac{r}{q}, \frac{p}{q}\right)} .
\end{align*}
Thus, for sufficiently large $q$ we obtain a contradiction, as desired. 
\end{proof}
\end{lemma}

We need the following lemma as well, which will not give us the full strength of the asymptotic estimates proved for  Carnot groups but still provide strong results.

\begin{lemma}[Lemma 2.3 in \cite{HarmanBook}]
\label{lemma:Harman}
Let $X$ be a measure space with measure $\mu$ such that $\mu(X)$ is finite. Let $\mathcal{E}_n$ be a sequence of measurable subsets of $X$ such that \( \sum_{n=1}^\infty \mu(\mathcal{E}_n) = \infty.\) Then the set $E$ of points belonging to infinitely many sets $\mathcal{E}_n$ satisfies 
\[ \label{eq:HarmanInequality}
 \mu(E) \ge \limsup_{N\to \infty} \left( \sum_{n=1}^N \mu(\mathcal{E}_n) \right)^2 \left( \sum_{n,m=1}^N \mu(\mathcal{E}_n \cap \mathcal{E}_m) \right)^{-1} .
\]
\end{lemma}

We require some number-theoretic lemmas before proceeding.

\subsection{Some number-theoretic lemmas}


We will use the standard norm on Gaussian integers defined by $N(\alpha) = |\alpha|^2$.

We will, as before, define the Gaussian totient function $\phi(\alpha)$ for $\alpha \in \Z[\ii]$ as being the number of invertible elements in $\Z[\ii]/\alpha\Z[\ii]$.  The function $\phi$ is multiplicative, so that $\phi(\alpha \beta) = \phi(\alpha)\phi(\beta)$ if $\gcd(\alpha,\beta)=1$. Also on primes $\pi\in \Z[\ii]$ we have that $\phi(\pi^k) = N(\pi^k) (1-N(\pi)^{-1})$. As such, we have the general form
\[\label{eq:phimultdef}
\phi(\alpha) = N(\alpha) \sideset{}{'}\prod_{\pi|\alpha} \left( 1- \frac{1}{N(\pi)}\right).
\]
Clearly $\phi(\alpha) \le  N(\alpha)$. 
The primed product here---and primed sums later---mean that we multiply (or add) only over terms that are non-negative reals or complex numbers that lie fully in the first quadrant. This removes the ambiguity of having several terms which are unit multiples of one another.

We define the M\"{o}bius function on the Gaussian integers as the multiplicative function $\mu$ defined on Gaussian primes $\pi$ by $\mu(\pi)=-1$ and $\mu(\pi^k) = 0$ if $k\ge 2$. This is analogous to the classical definition of the M\"{o}bius function for $\Z$. From \eqref{eq:phimultdef}, we can quickly derive the formula
\[\label{eq:phimudef}
\phi(\alpha) = \sum_{\delta\gamma = \alpha} N(\gamma) \mu(\delta) .
\]
Here, we assume that $\delta$ satisfies the conditions of the primed sum restriction above, although $\gamma$ may not.

For a given $\beta \in \Z[\ii]\setminus\{0\}$, consider the lattice $\beta \Z[\ii]$. The square tile formed with corners $0, \beta, \ii \beta, (1+\ii)\beta$ contains a complete residue system for $\beta$. (On the border we must be a little more careful, so we may take the sides going from $0$ to $\beta$ and from $0$ to $\ii \beta$ as being in our tile, but the other two sides as not in our tile. The only corner we allow to be in our tile is $0$ itself.) Thus, this tile contains exactly $\phi(\beta)$ Gaussian integers that are relatively prime to $\beta$. In fact, all the translates of this tile by elements of the lattice $\beta\Z[\ii]$ have this property. This is because if $\alpha$ is relatively prime to $\beta$ then so is $\alpha+\beta \gamma$ for any $\gamma\in \Z[\ii]$.

\begin{lemma}\label{lem:circleproblemvariant}
Let $\beta \in \Z[\ii]\setminus \{0\}$. Let $S$ be some set of reduced residue classes modulo $\beta$. Then we have
\(
\sum_{\substack{|\alpha|\le K\\ \alpha \bmod{\beta} \in S}} 1 = \pi |S| \frac{K^2}{|\beta|^2} + O\left( |S| \frac{K}{|\beta|}\right),
\)
provided $K> 2|\beta|$. The implicit constant is absolute.  If the sum is replaced with a primed sum, then a factor of $1/4$ would be added to the right-hand side.
\end{lemma}

\begin{proof}
In the case where $\beta=\pm 1$ or $\pm \ii$ and $S=\{0 \pmod{1}\}$, this is just the classical bounds on the Gauss circle problem. 

If $|\beta|>1$ and $S= \{ 0 \pmod{\beta} \}$, then this is counting the number of integer lattice points $\beta \Z[\ii]$ that are in the ball around the origin of radius $K$. However, by dividing through by $\beta$ we see that this is equivalent to the number of integers $\alpha \in \Z[\ii]$ with $|\alpha | \le K/|\beta|$. By the previous paragraph, this has the desired form.

Now we consider an arbitrary $|\beta|>1$ and arbitrary $S$. We again consider tiles formed by translates of the square with corners at $0,\beta,\ii \beta, (1+\ii) \beta$ by a point in the lattice $\beta \Z[\ii]$. We will call the tile with corners at $0,\beta,\ii\beta, (1+\ii)\beta$ the original tile.  Each such tile contains exactly $|S|$ points $\alpha$ such that $\alpha \pmod{\beta} \in S$. Moreover, each such tile can be uniquely identified with an integer point in the lattice $\beta\Z[\ii]$ that equals the point in the lattice required to translate the original tile to the given one. Moreover, this identifying point is also a corner of the tile. 

Since each tile has diameter equal to $\sqrt{2}|\beta|$, we see that the total number of tiles strictly contained in a ball around the origin of radius $K$ is at least the number of integer lattice points in $\beta \Z[\ii]$ that are in a ball around the origin of radius $K-\sqrt{2}|\beta|$. This in turn is equal to
\(
\pi \frac{1}{|\beta|^2} (K- \sqrt{2} |\beta|)^2 + O\left( \frac{K-\sqrt{2}|\beta|}{|\beta|}\right) = \pi \frac{K^2}{|\beta|^2} + O\left( \frac{K}{|\beta|}\right).
\)
Multiplying by $|S|$ to count the number of desired points within each tile gives a lower bound of the desired size. It was here that we used our assumption that $K>2|\beta|$ in order to have the asymptotics hold.

To obtain an upper bound, we note that each tile that intersects a ball around the origin of radius $K$ has its identifying point in a ball around the origin of radius $K+\sqrt{2}|\beta|$. A similar argument to the previous paragraph gives the desired bound.

For the final part, all the calculations are the same if we replace the full $|\alpha|\le K$ with a quarter-circle, with the exception of needing to worry about tiles that intersect the two radii of the quarter-circle. However, it is easy to show that the number of such terms is still $O(K/|\beta|)$.
\end{proof}

Recall that rational points on $\Sieg^1$ are given by $(r/q,p/q)$, $r,p,q\in \Z[\ii]$, $q\neq 0$ with $|r|^2-2 \Re(\overline{q}p)=0$. This distinction requires that $2$ divide $|r|^2$ and thus that $1+\ii$ divides $r$. We will let $\tilde{r}=r/(1+\ii)$. We will abuse notation  and say that $r/q$ is in lowest terms if $\tilde{r}/q$ itself is in lowest terms.

\begin{lemma}\label{lem:rationalcountbound}
Let $q\in\Z[\ii]$ satisfy $\Re(q),\Im(q)>0$ and $\gcd(\Re(q),\Im(q))=1$. Let $R>1$, and let $f(q)$ count the number of rational points $(r/q,p/q)$ with $r/q$ in lowest terms and $\Norm{(r/q,p/q)}<R$. Then $f(q) \asymp \phi(q)R^4$ provided $R$ is sufficiently large.
\end{lemma}

\begin{proof}
 Our goal will be to first count how many possible $r$ there can be for this $q$ and then count how many $p$ there can possibly be for a fixed pair $(q,r)$.

Let $q=a+b\ii$ be fixed and let $p=c+d\ii$ be variable. Using our definition of $\tilde{r}=r/(1+\ii)$, the equation $|r|^2-2\Re(\overline{q}p)=0$ can be rewritten as 
\(
|\tilde{r}|^2 = ac+bd.
\)
Since $(a,b)=1$ by assumption, this does not pose any restriction on $\tilde{r}$. 

We have that \(\Norm{\left( \frac{r}{q}, \frac{p}{q} \right)} = \left| \frac{p}{q}\right|^{1/2} = \sqrt[4]{\frac{1}{4} \left| \frac{r}{q}\right|^4+ \left|\Im\left( \frac{p}{q}\right)\right|^2} \asymp \max\left\{ \left| \frac{r}{q}\right| , \left| \Im \left(\frac{p}{q}\right)\right|^{1/2}\right\} \)

This imposes an additional restriction that $|\tilde{r}|\ll R |q|$. Our restrictions thus require that $(\tilde{r},q)=1$, $|\tilde{r}| \le MR|q|$ for some constant $M>0$ and that is all. We may assume that $M>2$ without loss of generality.  By Lemma \ref{lem:circleproblemvariant}, we have that the number of $\tilde{r}$ satisfying both those conditions is $ \asymp \phi(q) R^2$. (It was in order to apply this lemma that we assumed that $M>2$.)


Now suppose that $q=a+b\ii$ and $\tilde{r}$ are fixed, with $\tilde{r}$ satisfying the above calculations. We now wish to count how many possibilities for $p=c+d\ii$ there are. We have that
\[\label{eq:tilderequation}
|\tilde{r}|^2 = ac+bd.
\]
Since $c$ is completely dependent on $d$, it suffices to count the number of $d$.
Since $(a,b)=1$ we see that \eqref{eq:tilderequation} clearly has solutions.  In fact, if $d$ is any integer such that $bd \equiv |\tilde{r}|^2\pmod{a}$, then there exists a $c$ such that \eqref{eq:tilderequation} is satisfied. However in order for this equivalence relation to hold, we must have that $d$ belongs to a single, specific residue class modulo $a$.

Since we want $|p/q|^{1/2} \le R$, we have that $c^2+d^2\le R^4 |q|^2$. By solving for $c$ in \eqref{eq:tilderequation} and inserting this into our inequality, we get that 
\(
R^4 |q|^2 \ge \left( \frac{|\tilde{r}|^2-bd}{a}\right)^2 +d^2
\)
and solving for $d$ we see that it must lie in the interval
\(
\left[ \frac{b|\tilde{r}|^2 -a \sqrt{|q|^4 R^4-|\tilde{r}|^4}}{|q|^2},\frac{b|\tilde{r}|^2 +a \sqrt{|q|^4 R^4-|\tilde{r}|^4}}{|q|^2}   \right]
\)
This interval has length
\[\label{eq:dinterval}
\frac{2a \sqrt{|q|^4 R^4 - |\tilde{r}|^4}}{|q|^2}.
\]
If we want an upper bound on the possible number of $p$'s, we note that the square root is at most $|q|^2 R^2$ and thus $d$ lies in an interval of length at most $2a R^2$, and since $d$ must lie in residue class modulo $a$ we get that there are $\ll R^2$ possible choices for $d$ and thus for $p$, if $R$ is large.

On the other hand if we want a lower bound on the number of $p$'s, we must be more careful with our choices of $\tilde{r}$. To get a lower count, we may assume we restrict $\tilde{r}/q$ to have norm at most $R/2$. There are still $\asymp \phi(q) \cdot R^2$ possible values for $\tilde{r}$ provided $R$ is large enough. Now, the interval \eqref{eq:dinterval} has length at least 
\(
\frac{2a \sqrt{|q|^4 R^4 - |\tilde{r}|^4}}{|q|^2} \ge \frac{2a \sqrt{|q|^4 R^4 - \frac{|q|^4 R^4}{2^4}}}{|q|^2} \ge 2aR^2\sqrt{1-2^{-4}}
\) so there are $\gg R^2$ possible choices for $d$ and thus $p$. Combining these two estimates we see that there are $\asymp R^2$ choices for $p$.

In total, there are $\asymp \phi(q) \cdot R^4$ possible points.
\end{proof}

\begin{lemma}\label{lem:linearformbound}
Let $q,Q\in \Z[\ii]\setminus\{0\}$ and $A\ge 1$. The number of solutions to 
\(
0< |q x- Q y| \le A, \quad x,y\in \Z[\ii], \quad |x/Q|,|y/q| < R
\)
is bounded by $O(A^2 R^2)$ provided $R$ is sufficiently large.
\end{lemma}

\begin{proof}
Let us rewrite the first condition as $qx-Qy = a$ for some $a\in \Z[\ii]$ satisfying $0<|a|\le A$. This only has solutions if $g=\gcd(q,Q)$ divides $a$. In which case, we may write $q'=q/g$ and likewise define $Q'=Q/g$ and $a'=a/g$. So we want to count how many solutions there are to $q'x-Q'y =a'$. Clearly $y$ is completely dependent on $x$ and since we only want an upper bound, we only need to count the number of possible $x$'s. Note that $x$ must satisfy $q'x \equiv a' \pmod{Q'}$. This has a single solution (since $\gcd(q',Q')=1$) modulo $Q'$. By Lemma \ref{lem:circleproblemvariant}, we see that there are at most $\ll R^2 |Q|^2/ |Q'|^2 = R^2 g^2$ possible values for $x$ (and hence for pairs $(x,y)$) for this given $a$.

Now, again by Lemma \ref{lem:circleproblemvariant}, if $A>2g$, then we see that the number of $ a$ satisfying $0<|a|\le A$ and $g|a$ is bounded by $\ll A^2/g^2$. However, since we are assuming $a$ cannot be equal to $0$, this bound can also be seen to hold when $A\le 2g$. Combining these estimates gives the desired result.
\end{proof}

We note that in the next lemma, the full sum can be replaced with a primed sum without changing the statement of the lemmas, as it will only alter the expected magnitude by a multiple of $4$.

\begin{lemma}\label{lem:nalphacondiv}
Let $\epsilon>0$ and let $K$ be real. As $K$ goes to infinity, we have that the sum $ \sum_{0<N(\alpha)\le K} N(\alpha)^{-1-\epsilon}$ converges to a constant. Alternately, the sum $ \sum_{0 < N(\alpha)} N(\alpha)^{-1}$ diverges. In fact, $ \sum_{0 < N(\alpha) \le K } N(\alpha)^{-1} \asymp \log K$.
\end{lemma}

\begin{proof}
Consider the sum $ \sum_{A<|\alpha|\le 2A} N(\alpha)^{-1-\epsilon}$ where $\epsilon\ge 0$. (This allows us to consider both cases of the lemma simultaneously.) Each summand has size $\asymp A^{-2-2\epsilon}$. However, applying Lemma \ref{lem:circleproblemvariant} tells us that the number of terms $\alpha$ in this annulus is $\asymp A^2$ provided $A$ is sufficiently large, say $A>A_0$. Thus we have that this sum is $\asymp A^{-2\epsilon}$.

Now assume that $K$ is large. Let $j$ be a positive integer satisfying $2^j A_0 \le K < 2^{j+1} A_0$. Then by our work of the previous paragraph, we have that
\(
& \sum_{0 < N(\alpha) \le A_0} N(\alpha)^{-1-\epsilon} + \sum_{i=0}^{j-1}  \sum_{2^i A_0<|\alpha|\le 2^{i+1} A_0}  N(\alpha)^{-1-\epsilon} \\& \qquad \ll  \sum_{0 < N(\alpha)\le K} N(\alpha)^{-1-\epsilon} \\ &\qquad \ll  \sum_{0 < N(\alpha) \le A_0} N(\alpha)^{-1-\epsilon} +\sum_{i=0}^{j}  \sum_{2^i A_0<|\alpha|\le 2^{i+1} A_0}  N(\alpha)^{-1-\epsilon} 
\)
implies
\[\label{eq:nepsiloninequality}
M + \sum_{i=0}^{j-1} (2^i A_0)^{-2\epsilon} \ll  \sum_{0 < N(\alpha)\le K} N(\alpha)^{-1-\epsilon} \ll M+ \sum_{i=0}^{j} (2^i A_0)^{-2\epsilon},
\]
for some constant $M$. However, if $\epsilon >0$ then the sums on both sides of \eqref{eq:nepsiloninequality} converge, and since the sum in the middle consists of positive terms and is bounded for all $K$, it converges as $K$ goes to infinity.  On the other hand, if $\epsilon=0$, then the sums on both sides are on the order of $j$, which in turn is on the order of $\log K$, completing the proof.
\end{proof}

\begin{lemma}\label{lem:munonzero}
For $k\ge 2$, the series
\(
 \sideset{}{'}\sum_{\alpha\neq 0} \frac{\mu (\alpha)}{N(\alpha)^k}
\)
converges to a non-zero constant.
\end{lemma}

In this lemma, we are essentially proving the non-vanishing of the Dedekind zeta function for $\Z[\ii]$.

\begin{proof}
From Lemma \ref{lem:nalphacondiv}, we have that the series $\sum_{\alpha\neq 0} \mu (\alpha)/N(\alpha)^{-k}$ converges absolutely for $k\ge 2$.  So we may factor it into its Euler product:
\(
 \sideset{}{'}\sum_{\alpha\neq 0} \frac{\mu (\alpha)}{N(\alpha)^k} =  \sideset{}{'}\prod_{\pi} \left( 1-\frac{1}{N(\pi)^k} \right).
\)
Here the product on the right runs over all Gaussian primes. This is non-zero provided the limit of the products $\sideset{}{'}\prod_{N(\pi)<K} (1-N(\pi)^{-k})$ converges to something other than $0$. By taking logarithms this is just
\(
 \sideset{}{'}\sum_{N(\pi)<K}  \log (1-N(\pi)^{-k}),
\)
and we now need that this does not diverge to negative infinity. However, by Taylor's Theorem $|\log (1-x)| \asymp x$ for $|x|$ bounded away from $1$ (which is satisfied in all terms in this sum). Therefore,
\(
 \sideset{}{'}\sum_{N(\pi)<K}  \log (1-N(\pi)^{-k}) \asymp  \sideset{}{'}\sum_{N(\pi)<K} N(\pi)^{-k}
\) and by Lemma \ref{lem:nalphacondiv} we know that the latter sum converges even if we include all Gaussian integers, not just primes, thus this sum must converge as well.
\end{proof}

\begin{lemma}\label{lem:phiiislarge} 
There exists a constant $c>0$ such that if $K$ is sufficiently large, then $\phi(\alpha) \ge c N(\alpha)$ for at least $c\pi K^2$ of the $\alpha$ with $0<N(\alpha)\le K$.
\end{lemma}

\begin{proof} It suffices to prove this for $\alpha$ that are positive real or lying strictly in the first quadrant. 

This lemma is an easy corollary of showing that 
\[\label{eq:phiislargeeq}
\sideset{}{'}\sum_{0<N(\alpha)\le K} \frac{\phi(\alpha)}{N(\alpha)} \gg K^2,
\]
since if no constant $c>0$ satisfying the conditions of the lemma can be found, then the sum on the left-hand side of \eqref{eq:phiislargeeq} must be $o(K)$. (Here we would use that $\phi(\alpha)\le N(\alpha)$ for all $\alpha\in \Z[\ii]\setminus\{0\}$.)

By applying \eqref{eq:phimudef}, we see that
\(
\sideset{}{'}\sum_{0<N(\alpha)\le K} \frac{\phi(\alpha)}{N(\alpha)} &= \sideset{}{'}\sum_{0<N(\alpha) \le K} \frac{1}{N(\alpha)} \left( \sum_{\delta\gamma = \alpha} N(\gamma) \mu(\delta)\right)\\
&= \sideset{}{'}\sum_{0< N(\alpha)\le K} \sum_{\delta\gamma=\alpha} \frac{\mu(\delta)}{N(\delta)}\\
&= \sideset{}{'}\sum_{0< N(\delta)\le K} \frac{\mu(\delta)}{N(\delta)} \left( \sideset{}{'}\sum_{\substack{\delta|\alpha\\ N(\alpha)<K}} 1 \right).
\)
Here, in the internal sums on the first and second lines, we assume again that $\delta$ is a non-negative real or strictly within the first quadrant.

By applying Lemma \ref{lem:circleproblemvariant}, we have that 
\(
\sideset{}{'}\sum_{0<N(\alpha)\le K} \frac{\phi(\alpha)}{N(\alpha)} &= \sideset{}{'}\sum_{0< N(\delta)\le K} \frac{\mu(\delta)}{N(\delta)}\left( \frac{\pi}{4}\left(\frac{K}{N(\delta)}\right)^2 + O\left( \frac{K}{N(\delta)}\right)\right)\\
&= \frac{\pi}{4} K^2 \sideset{}{'}\sum_{0 < N(\delta)\le K} \frac{\mu(\delta)}{N(\delta)^3} + O\left( K \sum_{0<N(\delta)\le K} \frac{1}{N(\delta)^2}\right).
\)
By Lemma \ref{lem:munonzero}, the first sum converges to a non-zero constant $M$, while by Lemma \ref{lem:nalphacondiv}, the sum in the big-Oh term converges to a constant. Thus,
\(
\sideset{}{'}\sum_{0<N(\alpha)\le K} \frac{\phi(\alpha)}{N(\alpha)}  = \frac{\pi}{4} M K^2 (1+o(1)),
\)
and this  gives \eqref{eq:phiislargeeq} after another application of Lemma \ref{lem:circleproblemvariant}.
\end{proof}

\begin{prop}\label{prop:phisumasymp}
We have that the sum
\(
\sideset{}{^*}\sum_{0< N(\alpha)\le K} \frac{\phi(\alpha)}{N(\alpha)^2} \asymp \log K,
\)
where the starred sum runs over all $\alpha$ with $\Re(\alpha),\Im(\alpha)>0$ and $\gcd(\Re(\alpha),\\ \Im(\alpha))=1$.
\end{prop}

The difficulty with this proof is that even though we expect $\phi(\alpha)$ to be roughly of the size of $N(\alpha)$ for most $\alpha$, this might happen only when $\alpha$ whose real and imaginary parts have a factor in common, so we must be more careful in our estimates.

\begin{proof}
The upper bound follows directly from the fact that $\phi(\alpha)\le N(\alpha)$ together with Lemma \ref{lem:nalphacondiv}, so we need only prove the lower bound. We shall throughout this proof, keep to the standard convention that $\gcd(0,\alpha)= \alpha$.


Let $S_g(K)$ be the sum 
\(
\sideset{}{'}\sum_{\substack{0< N(\alpha)\le g^2 K\\ \gcd(\Re(\alpha),\Im(\alpha))=g}} \frac{\phi(\alpha)}{N(\alpha)^2},
\)
so that the starred sum equals $S_1(K)-1$, where the $-1$ term comes from $\alpha=1$. 
Note that we can write $S_g(K)$ as
\(
\sideset{}{'}\sum_{\substack{0< N(\alpha)\le K\\ \gcd(\Re(\alpha),\Im(\alpha))=1}} \frac{\phi ( g \alpha )}{N(g\alpha )^2}
\)
By using \eqref{eq:phimultdef}, we see that 
\(
\phi(g\alpha)&= N(g\alpha)\sideset{}{'}\prod_{\pi| g\alpha} \left( 1- \frac{1}{N(\pi)}\right) \\
&= g^2 \cdot N(\alpha)\sideset{}{'}\prod_{\pi| \alpha} \left( 1- \frac{1}{N(\pi)}\right) \sideset{}{'}\prod_{\substack{\pi|g\\ \pi\nmid \alpha}} \left( 1- \frac{1}{N(\pi)}\right) \\
&\le g^2 \cdot N(\alpha)\sideset{}{'}\prod_{\pi| \alpha} \left( 1- \frac{1}{N(\pi)}\right) =  g^2 \phi(\alpha).
\)
Here again, all products are over Gaussian primes $\pi$.
 Thus since $N(g\alpha)^2 = g^4 \cdot N(\alpha)^2$, we have that $S_g(K) \le S_1(K)/g^2$. In particular this implies that the sum $\sum_{g=1}^\infty S_g(K)$ converges absolutely and is bounded by $(\pi^2/6) S_1(K)$.

Let $S'(K)$ be the sum 
\[\label{eq:sprimek}
\sideset{}{'}\sum_{0< N(\alpha)\le K} \frac{\phi(\alpha)}{N(\alpha)^2} .
\]
Each summand $\alpha$ can be uniquely sorted into a sum of the form $S_g(K)$ where $g= \gcd(\Re(\alpha),\Im(\alpha))$. Thus $S'(K) \le \sum_{g=1}^\infty S_g(K)$, and thus $S'(K) \le (\pi^2/6) S_1(K)$.

By Lemma \ref{lem:phiiislarge}, there exists a constant $c<1$ such that $\phi(\alpha) \ge c N(\alpha)$ for $\pi c K^2$ of the $\alpha$ satisfying $N(\alpha)\le K$ provided $K$ is large; at least $\pi cK^2/4$ such $\alpha$ are either positive real or lie strictly in the first quadrant. Let us consider how many $\alpha$ satisfy $\sqrt{c} K/3< N(\alpha)\le K$, are either real or lie strictly in the first quadrant, and have $\phi(\alpha) \ge c N(\alpha)$. The larger circle $N(\alpha)\le K$ will contain at least $\pi  c K^2/4 (1+o(1))$ such $\alpha$'s (assuming the bare minimum satisfy the bound), while the inner circle $N(\alpha) \le \sqrt{c} K/3$ could contain up to $\pi c K^2/9 (1+o(1))$ such points (assuming all satisfy the bound). Thus we obtain at least $\pi K^2/9(1+o(1))$ such points in the annulus $\sqrt{c} K/3< N(\alpha)\le K$. 

For each such $\alpha$ in the annulus $\sqrt{c} K/3< N(\alpha)\le K$ satisfying $\phi(\alpha)\ge c N(\alpha)$, we have that the summand $\phi(\alpha) N(\alpha)^{-2}$ is of size $\asymp K^{-2}$ and there are $\asymp K^2$ such points in this annulus. Thus, the sum \eqref{eq:sprimek} on this annulus is bounded from below by a uniform constant. By splitting the whole sum $0< N(\alpha) \le K$ into a series of dyadic intervals, similar to the proof of Lemma \ref{lem:nalphacondiv}, we see that $S'(K)$ must be of size $\gg \log K$. Thus $S_1(K)\gg \log K$, proving the result.
\end{proof}

\subsection{Lower bound on $\alpha(\Sieg^n)$}

In this section, we prove Theorem \ref{thm:sieg1lowerbound}. For improved readability, we first prove the case $n=1$ and then indicate the changes necessary for the general case.
\begin{proof}[Proof of Theorem \ref{thm:sieg1lowerbound} for $n=1$]
We may assume, without loss of generality, that $C<1/2$.

We shall apply Lemma \ref{lemma:Harman}. We let $\mathcal{E}_q$ denote the set 
\(
\mathcal{E}_q := \sideset{}{^*}\bigcup_{r,p} B_{C/|q|}\left( \frac{r}{q},\frac{p}{q} \right),
\)
where the starred union denotes that we are running over all $r,p$ that satisfy the following:
\begin{itemize}
\item $\tilde{r}/q$ is in lowest terms, where $\tilde{r}=r/(1+\ii)$; and,
\item the point $(r/q,p/q)$ is inside a ball of radius $R$ with $R$  large enough that Lemmas \ref{lem:rationalcountbound} and \ref{lem:linearformbound} apply
\end{itemize} We  consider only those $q$ such that $\Re(q),\Im(q)>0$ and $\gcd(\Re(q),\Im(q))=1$; we only allow $q$ to be in the first quadrant in part because multiplying $p,r,q$ all by the same unit does not alter the rational point $(r/q,p/q)$, so we restrict $q$ to be in the first quadrant to ensure uniqueness.

We note the inequality in Lemma \ref{lemma:Harman} is intended to sum over positive integers, but since it makes use of a lim sup, we may replace the sum over integers with a sum over Gaussian integers $q$, $0<N(q)\le K$, with $q$ satisfying the above restrictions. Thus, to show that the set $E$ of points belonging to infinitely many sets $\mathcal{E}_q$ has positive measure, it suffices to prove that 
\[\label{eq:Siegdesiredeq}
 \limsup_{K\to \infty} \left( \sum_{0< N(q)\le K} \norm{\mathcal{E}_q} \right)^2 \left( \sum_{0< N(q),N(q')\le K} \norm{\mathcal{E}_q \cap \mathcal{E}_{q'}} \right)^{-1} > 0.
\]
If we can show \eqref{eq:Siegdesiredeq}, then the set of points $h$ satisfying 
\[\label{eq:refhnear}
d\left( h, \left( \frac{r}{q},\frac{p}{q}\right)\right) \le \frac{C}{|q|},
\]
for infinitely many rational points $(r/q,p/q)$ has positive measure, since this is even less restrictive than $h$ being in the set $E$; and therefore, by Lemma \ref{lem:SiegZO}, the set of points satisfying \eqref{eq:refhnear} for infinitely many rational points has full measure, proving the theorem.

To prove \eqref{eq:Siegdesiredeq}, we first want to know the size of $\norm{\mathcal{E}_q}$. Suppose we have two distinct points of the form $(r/q,p/q), (r'/q,p'/q)$ that are counted by the union that defines $\mathcal{E}_q$. The distance between these two points is given by
\[\label{eq:distinctdistance}
\left| \frac{p}{q} - \frac{r}{q}\overline{\left( \frac{r'}{q} \right)} + \overline{\left( \frac{p'}{q}\right)}\right|^{1/2} =\frac{|\overline{q}p - r\overline{r'} + \overline{p'}q|^{1/2}}{|q|} \ge \frac{1}{|q|}.
\]
Since the radius of each ball is at most half of this, the union is in fact a disjoint union. The volume of each ball is a constant times $(C/|q|)^{4}$ and, by Lemma \ref{lem:rationalcountbound}, the number of such points in the ball of radius $R$ is $\asymp \phi(q)$. (We are assuming $R$ is fixed so may absorb it into the asymptotic estimate.) Thus we have that $\norm{\mathcal{E}_q} \asymp \phi(q) |q|^{-4}$. 

By Proposition \ref{prop:phisumasymp}, we have that \[\label{eq:Siegestimateone}\sum_{0 < N(q) \le K} \norm{\mathcal{E}_q}\asymp \log K.\]

Now we consider $\mu(\mathcal{E}_q\cap \mathcal{E}_{q'})$. If $q=q'$ this is just $\norm{\mathcal{E}_q}$, which we have already determined the size of. We will therefore consider $q\neq q'$ for now. 

Consider two balls $B_{C/|q|}(r/q,p/q)$ and $B_{C/|q'|} (r'/q',p'/q')$ that have non-trivial overlap. By our lowest-terms condition and restriction that $q$ and $q'$ be in the first quadrant, we know that $r/q$ and $r'/q'$ cannot be the same. Thus, as in the proof for the Carnot model, we have
\(
0 < d\left( \left(\frac{r}{q},\frac{p}{q}\right), \left( \frac{r'}{q'},\frac{p'}{q'}\right)\right) \le \frac{C}{|q|}+ \frac{C}{|q'|} \le \max\{ |q|^{-1}, {|q'|}^{-1}\}.
\)
We can compare the usual distance with the infinity norm to get
\[\label{eq:distinctcenterdistance}
d\left( \left(\frac{r}{q},\frac{p}{q}\right), \left( \frac{r'}{q'},\frac{p'}{q'}\right)\right) \asymp \max\left\{  \left| \frac{r}{q}-\frac{r'}{q'} \right|, \left| \Im\left(\frac{p}{q}-\frac{p'}{q'} - \frac{r}{q} \overline{\frac{r'}{q'}} \right) \right|^{1/2} \right\}.
\]
In particular, we must have that 
\(
0 < \left| \frac{r}{q}-\frac{r'}{q'} \right| \le M\max\{ |q|^{-1}, {|q'|}^{-1}\},
\)
for some constant $M>0$. Alternately,
\(
0 < | q'r - r'q | \le M \max\{ |q/q'|, |q'/q|\}
\)
By Lemma \ref{lem:linearformbound}, there are $O(\max\{|q/q'|^2, |q'/q|^2\})$ possibilities for $r$ and $r'$. However, in the proof of Lemma \ref{lem:rationalcountbound} we noted that for a fixed $q,r,R$, there are a bounded number of $p$'s that satisfy $(r/q,p/q)$ in a ball of radius $R$ around the origin, and this estimate is uniform is $q$ and $r$, thus we can simply bound the possibilities for $p,p'$ by a constant. 

Since the size of the intersection is $O(\min\{|q|^{-4},|q'|^{-4}\})$ with the big-Oh constant dependent on $C$, we see that $\norm{\mathcal{E}_q\cap \mathcal{E}_{q'}} \ll |q|^{-2} |q'|^{-2}$. Therefore, by Lemma \ref{lem:nalphacondiv} we have that 
\[\label{eq:Siegestimatetwo}
\sum_{0 < N(q),N(q') \le K} \norm{\mathcal{E}_q\cap \mathcal{E}_{q'}} &\ll \left( \sum_{0 < N(q) \le K} \frac{1}{|q|^2}\right)^2 + \sum_{0 < N(q) \le K} \norm{\mathcal{E}_q} \\&\ll (\log K)^2,\notag
\]
again with the implicit constant being dependent on $C$. Combining \eqref{eq:Siegestimateone} and \eqref{eq:Siegestimatetwo} gives \eqref{eq:distinctdistance} and completes the proof.
\end{proof}

\begin{proof}[Proof of Theorem \ref{thm:sieg1lowerbound} for all $n$]

We will not provide the full proof of the extension of Theorem \ref{thm:sieg1lowerbound} to $n$ dimensions, but will explain the points where the proof differs from the $1$-dimensional case. 

First, recall that rational points in the $n$-dimensional Siegel model are $(r_1/q, r_2/q,  \dots, r_n/q, p/q)$ satisfying $|r_1|^2+|r_2|^2+\dots + |r_n|^2 - 2 \Re (\overline{q}p) = 0$. Unlike in the $1$-dimensional case, this does not require that $1+\ii$ divides anything. So instead we will make a new definition. We will say a rational point is \emph{good} if the following conditions hold
\begin{itemize}
\item $1+\ii$ divides $r_i$, $1\le i \le n$;
\item $\tilde{r}_i = r_i/(1+\ii)$ is relatively prime to $q$ for all $i$; and,
\item $q$ satisfies $\Re(q),\Im(q)>0$ and $\gcd(\Re(q),\Im(q))=1$.
\end{itemize}

The analog of Lemma \ref{lem:rationalcountbound} says that the number of good rational points with denominator $q$ in a ball of radius $R$ should be on the order of $\asymp \phi(q)^n R^{2n+2}$. We note that throughout the proof we can essentially replace all instances of $|r'|^2$ with $|r'_1|^2+|r'_2|^2+\dots+|r'_n|^2$ and this changes nothing in the counts of the number of possible $p$'s.

The next significant change would then be in the statement of Proposition \ref{prop:phisumasymp}, where we now want to show that
\(
\sum_{0<N(\alpha)\le K} \left( \frac{\phi(\alpha)}{N(\alpha)}\right)^n \cdot \frac{1}{N(\alpha)} \asymp \log K.
\)
The proof of the proposition is virtually unchanged.

The proof of the theorem itself has very little modification, but we just note some calculations that change. For starters, in place of \eqref{eq:distinctdistance}, we have the inequality
\(
&d\left( \left( \frac{r_1}{q}, \frac{r_2}{q}, \dots, \frac{r_n}{q}, \frac{p}{q}\right) , \left( \frac{r_1'}{q}, \frac{r_2'}{q}, \dots, \frac{r_n'}{q}, \frac{p'}{q}\right)\right)\\
&= \left| \frac{p}{q} -\frac{r_1}{q} \overline{\left( \frac{r'_1}{q}\right)}-\frac{r_2}{q} \overline{\left( \frac{r'_2}{q}\right)} - \dots -\frac{r_n}{q} \overline{\left( \frac{r'_n}{q}\right)} + \overline{\left( \frac{p'}{q}\right)} \right|^{1/2}\\
&= \left| \frac{p\overline{q} - r_1\overline{r'_1} - r_2 \overline{r'_2} - \dots - r_n \overline{r'_n} + q \overline{p'}}{q\overline{q}}\right|^{1/2}\\
&\ge \frac{1}{|q|}.
\)

Likewise, the right-hand side of \eqref{eq:distinctcenterdistance} would, instead of a single copy of $|r/q-r'/q'|$,  have a copy of $|r_i/q-r'_i/q|$ for each $i$ in $1\le i \le n$. Each of these would contribute $O(\max\{|q/q'|^2, |q'/q|^2\})$ to the count of the number of possible $r_i$'s, while again the count of the $p$'s would be bounded, thus giving a total of $O(\max\{|q/q'|^{2n}, |q'/q|^{2n}\})$ number of possible balls with distinct denominators $q,q'$ and non-trivial intersection. However, as the maximum size of this intersection would be $O(\min\{ |q|^{-2n-2}, |q'|^{-2n-2}\})$, this gives that $\norm{\mathcal{E}_q\cap \mathcal{E}_{q'}}$ would again be $\asymp |qq'|^{-2}$. The remaining details are the same.
\end{proof}

\subsection{Further improvements to the theorem}
\label{sec:moredetailsonquestions}
We would ideally like an asymptotic on the number of rational approximations in $\Sieg^n$, along the lines of Theorem \ref{thm:GeometricLowerBound}. We believe such a problem should be tractable with the methods used here. The main difficulty comes in Lemma \ref{lem:rationalcountbound}: in order to get a more precise asymptotic on the number of rational approximations, we need a more precise asymptotic on the number of rationals.

We can consider other variations of this problem. For $\Sieg^1$, one could consider a set $\mathcal{Q} \subset \Z[\ii]$, and ask whether when we should expect a point $(u,v)$ to have infinitely many rational approximations
\(
d\left( (u,v), \left( \frac{r}{q}, \frac{p}{q}\right)\right) \le \frac{C}{|q|},
\)
with $q\in \mathcal{Q}$. By our proof of Theorem \ref{thm:sieg1lowerbound}, we should expect infinitely many solutions for a positive measure set of $(u,v)$'s provided the following conditions hold: 
first, that $\mathcal{Q}$ is sufficiently dense so that
\(
\sum_{\substack{0< N(q) \le K\\ q \in \mathcal{Q}}} \norm{\mathcal{E}_q} \to \infty,
\)
and, second, that $\phi(q)$ is close to $N(q)$ for a sufficient number of $q\in \mathcal{Q}$ so that 
\(
\sum_{\substack{0< N(q) \le K\\ q \in \mathcal{Q}}} \norm{\mathcal{E}_q} \asymp \sum_{\substack{0< N(q) \le K\\ q \in \mathcal{Q}}} N(q)^{-1}.
\)
 This holds if, for example, we let $\mathcal{Q}$ be the set of Gaussian primes. However, our proof of Lemma \ref{lemma:SiegelRationalInvariance} means we cannot extend this positive measure set to a full measure set unless $k^2 q\in \mathcal{Q}$ for every $k\in \N$ and $q\in \mathcal{Q}$.

We remark that a variant of Theorem \ref{thm:GeometricLowerBound} is possible where one only considers $q$ to be a prime. In fact, the proof is simpler, as all the corresponding $\gcd(q,q')$ terms equal $1$. The resulting asymptotic would have $\log \log N$ in place of $\log N$.

\section{Badly approximable sets in $\Sieg^n$ are winning}\label{sec:badlyapprox}

We now study badly approximable points in Carnot groups, as well as in the Siegel model of $\Heis^n$. By Theorems \ref{thm:GeometricLowerBound} and \ref{thm:sieg1lowerbound}, the sets of badly approximable points have measure zero, as they are exactly the exceptions to those almost-everywhere results.

We first identify some points in non-commutative Carnot groups that deviate from the Diophantine exponent $\alpha(G)=(Q+1)/Q$. Note that the collection of such points is invariant under rational translations.
\begin{thm}
\label{thm:CarnotWeirdness}
Let $G$ be a non-commutative rational Carnot group. Then $G$ contains a linear subset $X$ with a strictly higher Diophantine exponent, and a linear subset $T$ with a strictly lower Diophantine exponent (with respect to the top-dimensional Hausdorff measure on each subset).
\begin{proof}
We work with the first Heisenberg group $\Heis^1$ in the Carnot model $\Carnot^1$, as the general case is analogous. As usual, we work with a weighted infinity metric $d$.

We claim the $x$-axis $X\subset \Carnot^1$ has Diophantine exponent equal to 2. Consider the embedding $\R\hookrightarrow \Carnot^1$ given by $x \mapsto (x,0,0)$. The embedding is an isometry, commutes with dilations, and sends rational numbers to rational points of $\Carnot^1$. In particular, Hausdorff 1-measure on $\R$ and $X$ agree. Lastly, we note that  the coordinate projection to $X$ is a Lipschitz map. Thus, we have that any approximation in $\R$ passes to $X\subset \Carnot^1$, and any approximation by a point of $\Carnot^1(\Q)$ of a point in $X$ gives a rational approximation of a point in $\R$. Thus, $\alpha(X,\Carnot^1)=\alpha(\R)=2$.

We now claim that the $t$-axis $T\subset \Carnot^1$ has Diophantine exponent equal to 1. The only difference is that distance along the $t$-axis is given by $d((0,0,t),(0,0,t'))= \lambda_2 \norm{t-t'}^{1/2}$, so that $T$ is isometric to the real line with the square root metric, which then gives $\alpha(T,\Carnot^1)=\alpha(\R)/2=1$. We note also that if $h\in T$ and $h'\in \Carnot^1$, then the coordinate projection of $h,h'$ to $T$ is Lipschitz, even though the projection may be distance-increasing if the restriction $h\in T$ is removed.
\end{proof}
\end{thm}

Note furthermore that the set $X\subset \Carnot^1$ has no $\height_{\Carnot^1}$-badly approximable points since \emph{any} point of $\R$ has irrationality exponent at least 2 (as one shows, for example, using real continued fractions). Thus, Theorem \ref{thm:GeometryIntro} reduces to:
\begin{thm}
\label{thm:ESK}
Let $K\subset \Sieg^n$ be a complete subset admitting a locally Ahlfors $\delta$-regular measure. Then $\BA_{\Sieg^n}\cap K$ is winning in $K$, and both sets have Hausdorff dimension $\delta$.
\end{thm}
We provide definitions and preliminary results for Theorem \ref{thm:ESK} in \S \ref{sec:metric}, and then prove it in \S \ref{sec:winning}.

\subsection{Schmidt games and metric spaces}
\label{sec:metric}
\begin{defi}[Schmidt game \cite{MR0195595}]
\label{defi:winning}
A set $A\subset X$ in a complete metric space $X$ in (Schmidt)-winning if for the following game between two players White and Black,  there exists $\alpha \in (0,1)$ such that for any $\beta \in (0,1)$, White has a winning strategy in $(\alpha, \beta)$ Schmidt game. The game is played as follows: Black chooses a (closed) metric ball $B_1\subset X$ of radius $\rho_1$. Then White chooses a metric ball $W_1\subset B_1$ of radius $\alpha \rho_1$. Then Black chooses a metric ball $B_2\subset W_1$ of radius $\beta \alpha \rho_1$, and so on. If, in the end $\cap_i B_i \subset A$, then White wins. Otherwise, Black wins.
\end{defi}

\begin{defi}
Let $(X,d)$ be a metric space with a measure $\mu$. One says that $\mu$ is (Ahlfors) $\delta$-regular if there exists $L\geq 1$ such that for any $\mu$-measurable subset $A\subset X$, one has $L^{-1}\text{diam}(A)^\delta \leq \mu(A) \leq L \text{diam}(A)^\delta$. 

If the volume estimate on $A$ is only satisfied for $\text{diam}(A)<r_0$ for some $r_0>0$, $(X, d, \mu)$ is said to be \emph{locally} (Ahlfors) $\delta$-regular.
\end{defi}

The Haar measure on the Heisenberg group $\Heis^n$ is $Q$-regular for $Q=2n+2$. Most reasonable subsets of $\Heis^n$ or $\R^n$, such as submanifolds, Cantor sets, and von Koch snowflakes, are also locally Ahlfors regular.

\begin{example}[Cantor set]
Let $X$ be a complete metric space, $0<\alpha<1$, $n>1$, and $C_0\subset X$ a ball of radius $r$. Select (if possible) $n$ disjoint balls of radius $\alpha r$ inside $C_0$ and set $C_1$ to be their disjoint union. Proceed recursively, replacing each ball of $C_i$ with $n$ balls of radius $\alpha^{i+1}r$, so that $C_{i+1}$ is a disjoint union of $n^{i+1}$ balls. The intersection $\mathcal C(\alpha, n, r) := \cap_{i=0}^\infty C_i$ is a \emph{Cantor set} and is homeomorphic to the standard middle-third Cantor set.
\end{example}

It is easy to compute the dimension of a Cantor set by a H\"older mapping to a corresponding Cantor set in $\R$. One has:
\begin{lemma}
Let $\mathcal C=\mathcal C(\alpha, n, r_0)$ be a Cantor set in a complete metric space. Then $\mathcal C$ has Hausdorff dimension $-\log(n)/\log(\alpha)$.
\end{lemma}

The following lemma (easily proven using a volume estimate argument) allows us to find balls in a locally Ahlfors regular space.
\begin{lemma}
\label{lemma:BobbySinger}
Let $X$ be a $\delta$-regular space with implicit constant $L$, $B_r(x)\subset X$, and $\beta\in(0,1]$. Then $B_r(x)$ contains at least $L^{-2}(6\beta)^{-\delta}$ disjoint balls of radius $\beta r$.
\end{lemma}

\begin{lemma}
Let $S$ be a winning set in a complete locally Ahlfors $\delta$-regular metric space $X$. Then it contains a Cantor set of Hausdorff dimension $\delta'$ arbitrarily close to $\delta$. In particular, $S$ has dimension $\delta$.
\begin{proof}
Note that a $\delta$-regular space has Hausdorff dimension at most $\delta$, so it suffices to demonstrate the lower bound.

Fix $(\alpha, \beta)$ so that White has a winning strategy and set $C_0=X$ (assuming without loss of generality that it is a metric ball of radius 1).

We proceed by analyzing Black's possible strategies. We assumed that every ball of radius $r$ contains $N$ disjoint balls of radius $\beta r$. For each choice of ball, White responds with a ball of radius $(\alpha \beta)r$. Set $C_1$ to be the disjoint union of these balls. At the next stage, we have $N^2$ disjoint balls to work with, and the choices are refined by White's strategy, yielding the disjoint union $C_2$. Proceeding inductively, we obtain a Cantor set $\mathcal C = \mathcal C(\alpha\beta, N, 1)$. Since White was always making moves towards a winning strategy, $\mathcal C\subset S$.

The Cantor set $\mathcal C$ we built has Hausdorff dimension $\delta'=\frac{-\log N}{\log(\alpha\beta)}$.  Recall that by the definition of winning sets we are free to vary $\beta$, so that our argument would be complete by taking $\beta \rightarrow 0$ if we could always choose $N \approx \beta^{-\delta}$. That is, it remains to show that every ball $B(x,r)$ of radius $r$ in $X$ contains (up to multiplicative error) $\beta^{-\delta}$ disjoint balls of radius $\beta r$. This is provided by Lemma \ref{lemma:BobbySinger}.
\end{proof}
\end{lemma}

\subsection{$\BA_{\Sieg^n}$ is winning}
\label{sec:winning}
We now prove Theorem \ref{thm:ESK} following \cite{MR2610978}, completing the proof of Theorem \ref{thm:GeometryIntro} .

\begin{remark}
A reader of \cite{MR2610978} should take in their proof of Lemma 4 that $R=\sqrt{\alpha\beta}^{-1}$ and $(2\epsilon+2r_1)<R^{-2}$. The results of \cite{MR2610978} must then be amended to speak of the properties of $\BA_\C$ rather than $\BA_{\C}^\epsilon$ for sufficiently small $\epsilon$, as the choice of $\epsilon$ depends on the game parameter $\beta$.
\end{remark}

\begin{thm}
\label{thm:winning}
Suppose $K\subset \Sieg^n$ admits a locally Ahflors regular measure. Then $\BA\cap K$ is a winning set.
\end{thm}
\begin{proof}
We will play the Schmidt game from Definition \ref{defi:winning} in the metric space $K$ with constants $(\alpha, \beta)$ and with the smaller set
$$\BA_{\Sieg^n}^\epsilon = \{ (u,v)\in \Sieg \st d((u,v), (r/q, p/q)>\epsilon/\norm{q}\text{ for all }r,p,q\in \Z[\ii]\}.$$
We make the following starting assumptions:
\begin{itemize}
\item We work, for simplicity, with $\Sieg^1$. 
\item The measure on $K$ is Ahlfors $\delta$-regular with implicit constant $L$.
\item $\alpha>0$ and satisfies $1 > L^{4}(12\alpha)^\delta$.
\item $\beta\in (0,1)$ is arbitrary.
\item $R=(\alpha\beta)^{-1}$.
\item $\epsilon \in (0,1)$ and $r_1\in (0,1)$ satisfy $2r_1 +2\epsilon<R^{-1}=\alpha\beta$.
\item Balls chosen by Black (of radius $r_i=R^{-i+1}r_1$) are denoted $B_i$.
\end{itemize}

We now provide a strategy for White such that $\cap_i B_i\in\BA_{\Sieg^n}^\epsilon$, and furthermore
\begin{align}
\label{eq:gamestrategy}
d\left(B_{i+1},\left(\frac{p}{q},\frac{r}{q}\right)\right)>\frac{\epsilon}{\norm{q}} \text{ if }\norm{q}<R^i.
\end{align}

By ignoring the first few moves, we may assume that the first ball $B_1$ chosen by Black has radius $r_1$ satisfying $2r_1+2\epsilon < R^{-1}$. From hereon, we avoid rational points on every turn. Note that condition \eqref{eq:gamestrategy} is trivial for the base case $i=0$, so it suffices to provide the inductive step.

Suppose Black has just chosen the ball $B_i$ of radius $r_i=r_1R^{-i+1}$, and we are to choose a ball inside $B_i$ of radius $\alpha r_i$. To satisfy \eqref{eq:gamestrategy}, we need to avoid all rational points near $B_i$ with denominator in the range
$$R^{i-1}\leq \norm{q} < R^{i}.$$

\begin{lemma}
\label{lemma:SchmidtUniqueness}
At most one such point exists.
\end{lemma}
\begin{proof}Suppose we have two points $(p/q,r/q)$ and $(p'/q', r'/q')$ that lie close to $B_i$. That is, there are $h,h'\in B_i$ such that
$$ d\left(h,\left(\frac{p}{q},\frac{r}{q}\right)\right)<\frac{\epsilon}{\norm{q}}
\text{ \ and \ } d\left(h',\left(\frac{p'}{q'},\frac{r'}{q'}\right)\right)<\frac{\epsilon}{\norm{q'}}.$$

By the triangle inequality and diameter of $B_i$, we compute that
$$d\left(
\left(\frac{r}{q}, \frac{p}{q}\right),
\left(\frac{r'}{q'}, \frac{p'}{q'}\right)\right) \leq \frac{\epsilon}{\norm{q}}+\frac{\epsilon}{\norm{q'}}+d(h,h') \leq 2\epsilon R^{-i+1} + 2 r_1 R^{-i+1} < R^{-i}.
$$

On the other hand, we compute 
\begin{align*}d\left(
\left(\frac{r}{q}, \frac{p}{q}\right),
\left(\frac{r'}{q'}, \frac{p'}{q'}\right)\right)&= 
\Norm{\left(-r/q, \overline{p/q}) * (r'/q', p'/q')\right)}\\
&=
\Norm{\left(-r/q+r'/q', \overline{p/q}+{p'/q'}+(\overline{-r/q})({r'/q'})\right)}\\
&=
\norm{\overline{p/q}+{p'/q'}+(\overline{-r/q})({r'/q'})}^{1/2}\\
&=\frac{\norm{\overline{p}q'+p'\overline{q}-\overline{r}{r'}}^{1/2}}{\norm{\overline{q}q'}^{1/2}} \geq\norm{qq'}^{-1/2}\geq R^{-i},
\end{align*}
where the last reduction follows from the $\overline{p}q'+p'\overline{q}-\overline{r}{r'}$ is a non-zero Gaussian integer --- as long as the two rational points are distinct.

The two calculations are contradictory, so the lemma is proven.
\end{proof}

We continue with the proof of Theorem \ref{thm:winning}. By Lemma \ref{lemma:SchmidtUniqueness}, we have at most one point to avoid as we choose ball for White. If there is no point to avoid, we choose arbitrarily. Otherwise, we denote the point by  $(p/q,r/q)$.

By Lemma \ref{lemma:BobbySinger}, $B_i$ contains $L^{-2}(6\alpha)^{-\delta}$ disjoint balls of radius $\alpha r_1 R^{-i+1}$. We would like to choose a ball all of whose points are at distance at least $\epsilon \norm{q}^{-1}$ from $(p/q,r/q)$. It would be impossible for us to make a choice if and only if  the disjoint balls were all within distance $\epsilon\norm{q}^{-1}+2\alpha r_1 R^{-i+1}$ of the point $(p/q,r/q)$. 

Now, the enlarged region we seek to avoid has volume at most
\begin{align*}L\cdot2^\delta \cdot (\epsilon\norm{q}^{-1}+2 \alpha r_1 R^{-i+1})^\delta &\leq L\cdot2^\delta \cdot (\epsilon R^{-i+1}+2 \alpha r_1 R^{-i+1})^\delta\\ &= L\cdot2^\delta \cdot (\epsilon+2 \alpha r_1)^\delta R^{\delta(-i+1)}.
\end{align*}
The disjoint balls we are choosing from each have volume at least
$$L^{-1}\cdot2^\delta \cdot (\alpha r_1 R^{-i+1})^\delta=L^{-1}\cdot2^\delta \cdot (\alpha r_1)^\delta R^{\delta(-i+1)}.$$
Dividing, the number of balls that fit inside the region we are avoiding is at most
$$L^{2}\cdot\left(\frac{\epsilon+2 \alpha r_1}{\alpha r_1}\right)^\delta,$$
and we can still choose a good ball if the estimate from Lemma  \ref{lemma:BobbySinger} is bigger:
$$L^{-2}\cdot(6\alpha)^{-\delta} > L^{2}\cdot\left(\frac{\epsilon+2 \alpha r_1}{\alpha r_1}\right)^\delta.$$
\begin{equation}
\label{eq:feasibility}
1 > L^{4}(6\alpha)^\delta\left(\frac{\epsilon+2 \alpha r_1}{\alpha r_1}\right)^{\delta}.
\end{equation}
Recall that we are given $L$, and then select $\alpha$ and after that $\epsilon$ and $r_1$. Suppose that we chose $\epsilon$ much smaller than $\alpha r_1$. Then the fraction in \eqref{eq:feasibility} is approximately equal to 2. This reduces the equation to one of our starting assumptions on $\alpha$, so we are able to pick the ball for White's $(i+1)$st move.
\end{proof}

\section{Acknowledgments}

Joseph Vandehey was supported in part by the NSF RTG grant DMS-1344994.  Anton Lukyanenko was supported by NSF RTG grant DMS-1045119. Part of the research was conducted by both authors while at MSRI, with support of the GEAR Network (NSF RNMS grants DMS 1107452, 1107263, 1107367). The authors would like to thank Ralf Spatzier discussions about this paper.

\bibliographystyle{amsplain}
\bibliography{citations}

\end{document}